\documentclass[ejs]{imsart}
\RequirePackage[OT1]{fontenc}
\RequirePackage{amsthm,amsmath,natbib}
\RequirePackage[colorlinks,citecolor=blue,urlcolor=blue]{hyperref}
\usepackage{amsmath}
\usepackage{graphicx,psfrag,epsf}
\usepackage{enumerate}
\usepackage{url}
\usepackage{pdfsync}
\usepackage{amssymb}
\usepackage[latin1]{inputenc}
\usepackage{amsthm}
\usepackage[english]{babel}
{\catcode `\@=11 \global\let\AddToReset=\@addtoreset}
\AddToReset{equation}{section}

\newtheorem{theorem}{Theorem}[section]
\newtheorem{lemma}{\bf Lemma}[section]

\newtheorem{proposition}{Proposition}[section]
\newtheorem{@definition}{\sc Definition}[section]

\newtheorem{@remark}{\sc Remark}[section]

\newtheorem{@example}{\sc Example}[section]

\newcommand{\beqn}{\begin{displaymath}}
\newcommand{\eeqn}{\end{displaymath}}
\newcommand{\beq}{\begin{equation}}  
\newcommand{\eeq}{\end{equation}}

\def\mathsf{\bf}
\def\N{\mathbb{N}}

\def\L{\mathbb{L}}
\def\R{\mathbb{R}}
\def\Z{\mathbb{Z}}

\def\E{\mathrm E}
\def\P{\mathrm P}

\def\text{\mbox}

\def\1{{\bf 1}}

\newcommand{\Var}{\mbox{Var}}

\newcommand{\Loi}{\mathcal{L}}

\def\egalloi{\renewcommand{\arraystretch}{0.5}
\begin{array}[t]{c}
\stackrel{{\Loi}}{=}
\end{array}\renewcommand{\arraystretch}{1}}

\def\limiteloin{\renewcommand{\arraystretch}{0.5}
\begin{array}[t]{c}
\stackrel{{\cal D}}{\longrightarrow} \\
{\scriptstyle n \rightarrow\infty}
\end{array}\renewcommand{\arraystretch}{1}}

\def\limiteproba{\renewcommand{\arraystretch}{0.5}
\begin{array}[t]{c}
\stackrel{{\cal P}}{\longrightarrow} \\
{\scriptstyle n \rightarrow\infty}
\end{array}\renewcommand{\arraystretch}{1}}

\def\limitepsn{\renewcommand{\arraystretch}{0.5}
\begin{array}[t]{c}
\stackrel{a.s.}{\longrightarrow} \\
{\scriptstyle n\rightarrow\infty}
\end{array}\renewcommand{\arraystretch}{1}}

\def\limitet{\renewcommand{\arraystretch}{0.5}
\begin{array}[t]{c}
\stackrel{}{\longrightarrow} \\
{\scriptstyle t\rightarrow\infty}
\end{array}\renewcommand{\arraystretch}{1}}

\def\limitet0{\renewcommand{\arraystretch}{0.5}
\begin{array}[t]{c}
\stackrel{}{\longrightarrow} \\
{\scriptstyle t\rightarrow 0}
\end{array}\renewcommand{\arraystretch}{1}}

\arxiv{arXiv:0000.0000}

\startlocaldefs
\numberwithin{equation}{section}
\theoremstyle{plain}

\theoremstyle{remark}

\endlocaldefs

\begin{document}

\begin{frontmatter}
\title{Asymptotic behavior of the Laplacian quasi-maximum likelihood estimator of affine causal processes}
\runtitle{Asymptotic behavior of the Laplacian QMLE}

\begin{aug}

\author{\fnms{Jean-Marc} \snm{Bardet}\thanksref{a,e1}\ead[label=e1,mark]{bardet@univ-paris1.fr}}
\author{\fnms{Yakoub} \snm{Boularouk}\thanksref{b,e2}\ead[label=e2,mark]{y.boularouk@centre-univ-mila.dz}}
\and
\author{\fnms{Khedidja} \snm{Djaballah}\thanksref{b,e3}%
\ead[label=e3,mark]{khdjeddour@hotmail.com}}

\address[a]{S.A.M.M., Universit\'{e} Panth\'eon-Sorbonne, 90, rue de Tolbiac, 75634, Paris, France.
\printead{e1}}

\address[b]{Universit\'{e} des Sciences et de la Technologie Houari Boum\'{e}diene, Alger, Alg\'{e}rie.
\printead{e2}, \printead{e3}}

\runauthor{Bardet {\it et al.}}

\affiliation{Some University and Another University}

\end{aug}

\begin{abstract}
~We prove the consistency and asymptotic normality of the Laplacian Quasi-Maximum Likelihood Estimator (QMLE) for a general class of causal time series including ARMA, AR($\infty$), GARCH, ARCH($\infty$), ARMA-GARCH, APARCH, ARMA-APARCH,..., processes. We notably exhibit the advantages (moment order and robustness) of this estimator compared to the classical Gaussian QMLE. Numerical simulations confirms the accuracy of this estimator.
\end{abstract}



\end{frontmatter}

\section{Introduction}
This paper is devoted to establish the consistency and the asymptotic normality of a parametric estimator for a general class of time series. This class was already defined and studied in \cite{DW}, \cite{BW} and \cite{BKW}. Hence, we will consider an observed sample $(X_1,\cdots,X_n)$ where $(X_t)_{t\in \Z}$ is a solution of the following equation:
\begin{equation}\label{eq::sys}
     X_t =M_{\theta_0}(X_{t-1},X_{t-2}, \cdots) \,   \zeta_t +  f_{\theta_0}(X_{t-1},X_{t-2}, \cdots), \qquad  t \in Z,
\end{equation}
where
\begin{itemize}
\item $\theta_0\in \Theta \subset \R^d$, $d\in \N^*$, is an unknown vector of parameters,  also called the "true" parameters;
\item $(\zeta_t)_{t\in \Z}$ is a sequence of centred independent identically distributed random variables (i.i.d.r.v.) with symmetric probability distribution, {\it i.e.} $\zeta_0 \egalloi -\zeta_0$, satisfying $\E [|\zeta_0|^r]<\infty$ with $r\geq 1$ and $\E [|\zeta_0|]=1$. If $r\geq 2$, denote $\sigma_\zeta^2=\Var(\zeta_0)$;
\item $(\theta,(x_n)_{n\in \N}) \to M_\theta((x_n)_{n\in \N}) \in (0,\infty)$ and $(\theta,(x_n)_{n\in \N}) \to f_\theta((x_n)_{n\in \N})\in \R$ are two known applications.
\end{itemize}
For instance, if $M_{\theta_0}(X_{t-1},X_{t-2}, \cdots)=1$ and $f_{\theta_0}(X_{t-1},X_{t-2}, \cdots)=\alpha_0 \, X_{t-1}$ with $|\alpha_0|<1$ then $(X_t)$ is a causal AR($1$) process. In \cite{DW} and \cite{BW}, it was proved that all the most famous stationary time series used in econometrics, such as ARMA, AR($\infty$), GARCH, ARCH($\infty$), TARCH, ARMA-GARCH processes can be written as a causal stationary solution of (\ref{eq::sys}). \\
In \cite{BW}, it was also established that under several conditions on $M_\theta$, $f_\theta$ and if $\E [|\zeta_0|^r]$ with $r\geq 2$, the usual Gaussian Quasi-Maximum Likelihood Estimator (QMLE) of $\theta$ is strongly consistent and when $r\geq 4$ it is asymptotically normal. This estimator was first defined by \cite{W} for ARCH processes, and the asymptotic study of this estimator was first obtained by \cite{L} for GARCH$(1,1)$ processes, \cite{BHK} for GARCH$(p,q)$ processes, \cite{FZ2004} for ARMA-GARCH processes, \cite{SM} for general heteroskedastic models, and \cite{RZ} for ARCH$(\infty)$ processes. The results of \cite{BW} devoted to processes satisfying almost everywhere (\ref{eq::sys}) as well as its multivariate generalisation, provide a general and unified framework for studying the asymptotic properties of the Gaussian QMLE. \\
However, the definition of the Gaussian QMLE is explicitly obtained with the assumption that $(\zeta_t)$ is a Gaussian sequence and even if it could be applied when the probability distribution of $(\zeta_t)$ is non-Gaussian, it keeps some drawbacks of this initial assumption. Indeed, the computation of this estimators requires the minimization of a least squares contrast (typically $\sum_{t=1}^n M_\theta^{-2} (X_t-f_\theta)^2$) and this induces that $r=2$ is required for the consistency and $r=4$ for the asymptotic normality (and therefore confidence intervals or tests). For numerous real data such requirement is sometimes too strong (for instance the kurtosis of economic data is frequently considered as infinite). Moreover, such estimator is not robust to potential outliers. Hence, the reference probability distribution of $(\zeta_t)$ could be a Laplace one and this allows to avoid both these drawbacks. Roughly speaking this choice implies to minimize a least absolute deviations contrast (typically $\sum_{t=1}^n M_\theta^{-1} |X_t-f_\theta|$) instead of the previous least squares contrast. And therefore, $r=1$ will be sufficient for insuring the strong consistency of this Laplacian-QMLE, while only $r=2$ is required for the asymptotic normality (see below). \\
Such probability distribution choice is not new since this leads to a Least Absolute Deviations (LAD) estimation. Hence, for ARMA processes, \cite{DD} proved the consistency and asymptotic normality of the LAD estimator. For ARCH or GARCH processes, the same results concerning the LAD estimator were already established by \cite{PY}, while \cite{BH} proved the consistency and asymptotic normality of the Laplacian-QMLE. \cite{NS} considered also the estimator for other conditional heteroskedasticity models. Recently, \cite{FLZ} proposed a two-stage non-Gaussian-QML estimation for GARCH models and \cite{FZ2015} proposed an alternative one-step procedure, based on an appropriate non-Gaussian-QML estimator, the asymptotic properties of both these approaches were studied. \\
In this paper we unify all these studies of the Laplacian-QMLE in a simple framework, {\it i.e.} causal stationary solutions of (\ref{eq::sys}). This notably allows to obtain known results on ARMA or GARCH but also to establish for the first time the consistency and the asymptotic normality of the Laplacian-QMLE for APARCH, ARMA-GARCH, ARMA-ARCH$(\infty)$ and ARMA-APARCH processes. \\
Numerical Monte-Carlo experiments were realized to illustrate the theoretical results. And the results of these simulations are convincing, especially when the accuracy of Laplacian-QMLE is compared with the one of Gaussian-QMLE: except for Gaussian distribution of $(\zeta_t)$, the Laplacian-QMLE provides a sharper estimation than the Gaussian-QMLE for all the other probability distributions we considered. This is notably the case, and this is not a surprise, for a Gaussian mixing which mimics the presence of outliers. This provides an effective advantage of the Laplacian QMLE compared to the Gaussian QMLE.\\
~\\
The following Section \ref{DefAss} will be devoted to provide the definitions and assumptions. In Section \ref{Main} the main results are stated with numerous examples of application, while Section \ref{Simu} presents the results of Monte-Carlo experiments and Section \ref{Proofs} contains the proofs.

\section{Definition and assumptions} \label{DefAss}

\subsection{Definition of the estimator}
Let $(X_1,\cdots,X_n)$ be an observed trajectory of $X$ which is an a.s. solution of (\ref{eq::sys}) where $\theta \in \Theta \subset \R^d$ is unknown. For estimating $\theta$ we consider the log-likelihood of $(X_1,\cdots,X_n)$ conditionally to $(X_0,X_{-1},\cdots)$. If $h$ is the probability density (with respect to Lebesgue measure) of $\zeta_0$, then, from the affine causal definition of $X$, this log-likelihood can be written:
$$
\log \big (L_\theta(X_1,\cdots,X_n)\big )= \sum_{t=1}^n \log \Big (\frac 1 {M_\theta^t} \, h\Big ( \frac {X_t-f_\theta^t}{M_\theta^t} \Big )\Big )
$$
where $M_\theta^t:=M_\theta (X_{t-1},X_{t-2},\cdots)$ and $f_\theta^t:=f_\theta (X_{t-1},X_{t-2},\cdots)$, with the assumption that $M_\theta^t>0$. However, $M_\theta^t$ and $f_\theta^t$ are generally not computable since $X_0,X_{-1},\ldots$ are unknown. Thus, a quasi-log-likelihood is considered instead of the log-likelihood and it is defined by:
$$
\log \big (QL_\theta(X_1,\cdots,X_n)\big )= \sum_{t=1}^n \log \Big (\frac 1 {M_\theta^t} \, h\Big ( \frac {X_t-\widehat f_\theta^t}{\widehat M_\theta^t} \Big )\Big ),
$$
with $ \widehat
f_\theta^t:=f_{\theta}(X_{t-1},\ldots,X_1,u)$ and $\widehat
M_\theta^t:=M_{\theta}(X_{t-1},\ldots,X_1,u)$ , where
$u = (u_n)_{n\in N}$ is a finitely non-zero sequence $(u_n)_{n\in N}$. The choice of $(u_n)_{n\in N}$ does not
have any consequences on the asymptotic behaviour of  $L_n$, and $(u_n)$ could typically be chosen as a sequence of zeros. Finally, if it exists, a Quasi-Maximum Likelihood Estimator (QMLE) is defined by:
$$
\widetilde \theta_n:= \mbox{Argmax}_{\theta \in \Theta} \log \big (QL_\theta(X_1,\cdots,X_n)\big ).
$$
Usually, the "instrumental" probability density $h$ is the Gaussian density, {\it i.e.} 
$$
h(x)=\frac 1 {\sqrt{2 \pi }}\, e^{-\frac 1 2 \, x^2} \quad\mbox{for $x \in \R$}
$$
and this provides the Gaussian-QMLE of $\theta$. \\
Here, we chose as instrumental probability density the Laplacian density, {\it i.e.},
\begin{equation}\label{densL}
h(x)= \frac{1}{2}e^{-|x|}\qquad \mbox{for $x \in \R$},
\end{equation}
and this implies  $\E\big [|\zeta_0| \big ]=1$.\\
Therefore, we respectively define the Laplacian-likelihood and Laplacian-quasi-likelihood   by:
\begin{eqnarray}\label{eq::ss}
L_n(\Theta)&=&-\sum_{t=1}^n  q_t(\Theta) \qquad \mbox{with}\qquad q_t(\Theta)= \log{|M_\theta^t|}+|M_\theta^t|^{-1}  |X_t-f^t_{\theta}| \\
\label{eq::QMLE}
\widehat L_n(\theta)&=&-\sum_{t=1}^n \widehat q_t(\theta) \qquad \mbox{with}\qquad \widehat q_t(\theta):=\log{|\widehat{M}_\theta^t|}+|\widehat{M}_\theta^t|^{-1}  |X_t-\widehat{f}^t_{\theta}|.
\end{eqnarray}
Hence, if it exists, a Laplacian-QMLE $\widehat{\theta}_n$ is a maximizer of $\widehat{L}_n$:
\begin{equation*}
\widehat{\theta}_n:= \arg\max_{\theta \in \Theta}  \widehat{L}_n(\theta).
\end{equation*}
We restrict the set $\Theta$ in such a way that a stationary solution $(X_t)$ of order $1$ or $2$ of \eqref{eq::sys} exists. Additional conditions are also required for insuring the consistency and the asymptotic normality of $\widehat{\theta}_n$. More details are given now.

\subsection{Existence and stationarity}
As it was already done in \cite{DW} and \cite{BW}, several Lipschitz-type inequalities on $f_\theta$ and $M_\theta$ are required for obtaining the existence and $r$-order stationary ergodic causal solution of \eqref{eq::sys}. \\
First, denote $\|g_\theta \|_\Theta=\sup_{\theta \in \Theta} \|g_\theta \|$ with $m\in \N^*$ and $\|\cdot \|$ the usual Euclidean norm (for vectors or matrix). Now, let us introduce the generic symbol $K$ for any of the functions
$f$ or $M$. For $k=0,\, 1,\, 2$ and some  subset $\Theta$ of $\R^d$, define a Lipschitz assumption on function $K_\theta$:\\
\vspace{0.2cm}
\newline {\bf Assumption (A$_k$($K,\Theta$))} $\forall x\in\R^{\infty}$, $\theta \in \Theta \mapsto K_\theta(x) \in {\cal C}^k(\Theta)$ and
$\partial_\theta^k K_\theta$ satisfies $\big \| \partial_\theta^k K_\theta(0) \big \|_\Theta <\infty$ and there exists a sequence
$\big(\alpha^{(k)}_j(K,\Theta)\big)_{j}$ of nonnegative numbers such that $\forall x$, $y\in\R^{\N}$
\begin{eqnarray*}\big \|\partial_\theta^k K_\theta(x) -\partial_\theta^k K_\theta(y) \big \|_\Theta \le \displaystyle  \sum_{j=1}^\infty
\alpha^{(k)}_j(K,\Theta)
|x_j-y_j|,~\mbox{with}~\sum_{j=1}^\infty \alpha^{(k)}_j(K,\Theta)<\infty
\end{eqnarray*}
~\\
For ensuring a stationary $r$-order solution of \eqref{eq::sys}, for $r\ge1$, define the set
\begin{multline*}
\Theta(r):=\Big\{\theta\in\R^d,~ \mbox{{\bf(A$_0$($f,\{\theta\}$))} and {\bf(A$_0$($M,\{\theta\}$))} hold}, \\
\sum_{j=1}^\infty \alpha_j^{(0)}(f,\{\theta\})+\left(\E[|\zeta_0|^r ]\right)^{1/r}\sum_{j=1}^\infty \alpha_j^{(0)}(M,\{ \theta\})<1
\Big\}.
\end{multline*}
Then, from \cite{DW}, we obtain:
\begin{proposition}\label{stationarity}
If $\theta_0 \in \Theta(r)$ for some $ r\geq 1$, then there exists a unique
causal ($X_t$ is independent of $(\zeta_i)_{i  > t}$  for $t \in \Z$) solution $X$ of (\ref{eq::sys}), which is stationary, ergodic and satisfies $\E \big [|X_0|^r \big ] < \infty$ .
\end{proposition}
\noindent The following lemma insures that if a process $X$ satisfies Proposition \ref{stationarity}, a causal predictable ARMA process with $X$ as innovation also satisfies Proposition \ref{stationarity}. We first provide the classical following notion for  a sequence $(u_n)_{n\in \N}$ of real numbers:
\begin{eqnarray*}
\nonumber &\mbox{$(u_n)_{n\in  \N}$ is an exponentially decreasing sequence (EDS)}& \\  \nonumber &\Longleftrightarrow& \\ \label{expo} &\mbox{there exists $\rho \in [0,1[$ such as $u_n=\mathcal{O}(\rho^n)$ when $n\to \infty$}.&
\end{eqnarray*}
\begin{lemma}\label{cond2}  Let $X$ be a.s. a causal stationary
solution of \eqref{eq::sys} for $\theta_0 \in \R^d$. 
Let $\widetilde X$ be such as $\widetilde X_t=\Lambda_\beta(L) \ X_t$ for $t \in \Z$ with $\Lambda_{\beta_0}(L)=P^{-1}_{\beta_0}(L) \, Q_{\beta_0}(L)$ where $(P_{\beta_0},Q_{\beta_0})$ are the coprime polynomials of a causal invertible ARMA$(p,q)$ processes with a vector of parameters $\beta_0\in \R^{p+q}$. Denote $\Lambda_{\beta_0}^{-1}(x)=Q^{-1}_{\beta_0}(x) \, P_{\beta_0}(x)=1+\sum_{j=1}^\infty \psi_j(\beta_0) x^j$. Then $\widetilde X$ is a.s. a causal stationary solution of the equation
$$
\widetilde X_t =\widetilde M _{\widetilde \theta_0}\big ((\widetilde X_{t-i})_{i\geq 1}\big ) \, \zeta_t +\widetilde f _{\widetilde \theta_0}\big((\widetilde X_{t-i})_{i\geq 1})\big ) \quad \mbox{for $t\in \Z$},
$$
where $\widetilde f _{\theta_0}$ and $\widetilde M _{\theta_0}$ are given in \eqref{Mftilde} and $\widetilde \theta_0=(\theta_0,\beta_0)$. Moreover, for $i=0, \,1, \, 2$ and with $K=f$ or $M$ and $\widetilde K=\widetilde f$ or $\widetilde M$,
\begin{itemize}
\item if $\alpha_j^{(i)}(K,\{\theta_0\})=O(j^{-\beta})$ and $\beta>1$, then $\alpha_j^{(i)}(\widetilde K,\{\widetilde \theta_0\})=O(j^{-\beta})$;
\item if $\alpha_j^{(i)}(K,\{\theta_0\})$ is EDS, then $\alpha_j^{(i)}(\widetilde K,\{\widetilde \theta_0\})$ is EDS.
\end{itemize}
\end{lemma}

\subsection{Assumptions required for the convergence of the Laplacian-QMLE}
The Laplacian-QMLE could converge and be asymptotically Gaussian but this requires
some additional assumptions on $\Theta$ and functions $f_\theta$ and $M_\theta$:
\begin{itemize}
    \item {\bf Condition C1} (Compactness) $\Theta$ is a compact set.
  \item {\bf Condition C2} (Lower bound of the conditional variance) There exists a deterministic constant $\underline{M} > 0$ such that for all $\theta \in\Theta$ and $x \in \R^\N$, then  $M_\theta(x)>\underline{M}$.
  \item {\bf Condition C3} (Identifiability) The functions $M_\theta$ and $f_\theta$ are such that:
 for all $\theta_1,~\theta_2 \in \Theta$, then $M_{\theta_{1}}=M_{\theta_2}$ and $f_{\theta_{1}}=f_{\theta_2}$ implies that $\theta_1=\theta_2$.
\end{itemize}

\section{Consistency and asymptotic normality of the estimator}\label{Main}

\subsection{Consistency and asymptotic normality}
First we prove the strong consistency of a sequence of Laplacian-QMLE for a solution of (\ref{eq::sys}). The proof of this theorem, is postponed in Section \ref{Proofs}, as well as the other proofs.
\begin{theorem}\label{SC}
Assume Conditions $C1$, $C2$ and $C3$ hold and $\theta_0 \in \Theta(r)\cap \Theta$ with $r\geq 1$. Let $X$ be the stationary solution of (\ref{eq::sys}). If ${ (A_0(f,\Theta))}$ and ${ (A_0(M,\Theta))}$ hold with
\begin{equation} \label{decSC}
\alpha_j^{(0)}(f,\Theta)+\alpha_j^{(0)}(M,\Theta)=\mathcal{O}(j^{-\ell})\ \  \mbox{for some}~   \ell > \frac 2 {\min(r\, , \, 2)}
\end{equation}
then a sequence of Laplacian-QMLE  $(\widehat{\theta}_n)_n$ strongly converges,  that is $\widehat{\theta}_n   \limitepsn \theta_0$.
\end{theorem}
\noindent Of course, the conditions required for this strong consistency of a sequence of Laplacian-QMLE are almost the same than the ones required for the strong consistency of a sequence of Gaussian-QMLE except that $r \in [1,2)$ is proved to be possible in Theorem \ref{SC} and not in case of Gaussian-QMLE (see \cite{BW}). Moreover, if $r=2$, the condition \eqref{decSC} on Lipshitzian coefficients is weaker for Laplacian-QMLE than for Gaussian-QMLE. As we will see below, many usual time series can satisfy the assumptions of Theorem \ref{SC}; for example, an AR$(\infty)$  process can be defined for satisfying the strong consistency of Laplacian-QMLE while the conditions given in \cite{BW} do not ensure the strong consistency of Gaussian-QMLE. \\
~\\
Now we state an extension of Theorem 1 established in \cite{DD} which will be an essential step of the proof of the asymptotic normality of the estimator.
\begin{theorem}\label{Davis}
Let $(Z_t)_{t\in \Z}$ be a sequence of i.i.d.r.v such as $\Var(Z_0)=\sigma^2<\infty$, with common distribution function which is symmetric ($F(-x)=1-F(x)$ for $x \in \R$) and is continuously differentiable in a neighborhood of $0$ with derivative $f(0)$ in $0$. Denote ${\cal F}_t=\sigma (Z_t,Z_{t-1},\cdots)$ for $t\in \Z$ and let $(Y_t)_{t\in \Z}$ and $(V_t)_{t\in \Z}$ two stationary processes adapted to  $({\cal F}_t)_t$ and such as $\E \big [ Y_0^2 V_0^2 \big ]<\infty$. Then
 \begin{equation}\label{dd}
\sum_{t=1}^n V_{t-1} \big ( | Z_t- n^{-1/2} Y_{t-1} | - |Z_t  | \big ) \limiteloin {\cal N} \Big ( f(0) \,  \E \big [ V_0 Y_0^2  \big ] \, , \, \E \big [ V_0^2 Y_0^2  \big ]\Big )
\end{equation}
\end{theorem}
\noindent Then, the asymptotic normality of the Laplacian-QMLE can be established using additional assumptions:
\begin{theorem}\label{AN}
Assume that $\theta_0\in \stackrel{\circ}{ \Theta} \cap \Theta(r)$ where $r \geq 2$ and
$\stackrel{\circ}{ \Theta}$ denotes the interior of $\Theta$. Let $X$ be the stationary solution of
the equation \eqref{eq::sys}. Assume that the conditions of
Theorem \ref{SC} hold and for $i=1,2$, assume ${ (A_i(f,\Theta))}$ and ${ (A_i(M,\Theta))}$ hold.
Then, if the cumulative probability function of $\zeta_0$ is continuously differentiable in a neighborhood of $0$ with derivative $g(0)$ in $0$ and if matrix $\Gamma_F$ or $\Gamma_M$, defined in \eqref{Gamma},  are definite positive symmetric matrix, then
\begin{eqnarray}\label{tlcqmle}
\sqrt{n}\big (\widehat\theta_n-\theta_0\big ) \limiteloin {\cal
N}_d\Big (0 \ , \ \big (  \Gamma_M +2g(0)\, \Gamma_F \big )^{-1} \big ( \big ( \sigma^2_\zeta -1 \big ) \, \Gamma_M + \Gamma_F \big )
\big (  \Gamma_M +2g(0)\, \Gamma_F \big )^{-1}\Big ).
\end{eqnarray}
\end{theorem}
\noindent As it was already proved for the median estimator (see \cite{V}) or for least absolute deviations estimator of ARMA process (see \cite{DD}), it is not surprising that the probability density function $g$ of the white noise $(\zeta_i)_i$ impacts the asymptotic covariance of \eqref{tlcqmle}. However, when $f_\theta=0$, this is not such the case and this is what happens for GARCH processes see \cite{FLZ}  where the probability density $g$ does not appear in the asymptotic covariance.
\subsection{Comments on these limit theorems}
Essentially, these limit theorems could appear close or even very close to the results of $3$ other references we chronologically list below but also from which we highlight the differences:
\begin{itemize}
\item The first related paper is \cite{DD} which is cited many times. The framework of this paper is restricted to the LAD (similar to the Laplacian-QMLE) of the parameters of ARMA$[p,q]$ process or residuals of least-square estimation with ARMA$[p,q]$ errors. If the framework \eqref{eq::sys} is clearly more general since it includes for instance GARCH, ARMA-GARCH or APARCH process, the proof we used for establishing the asymptotic normality of the Laplacian estimator is clearly inspired by the one of \cite{DD}. Thus our results could appear as extensions of this paper. 
\item The second and certainly closest paper, \cite{BW}. The considered framework is exactly the same, {\it i.e.} general causal affine models and the estimation method is the same, {\it i.e.} the quasi-maximum likelihood estimation (QMLE). However in \cite{BW} the QMLE is based on an "instrumental" Gaussian density instead of a Laplacian one. As it is such the case for instance by comparing quantile with least square regression, this implies three main differences:
\begin{enumerate}
\item The moment conditions $r$ of both the limit theorems (strong consistency and asymptotic normality) are weaker with Laplacian QMLE than with the Gaussian one. Indeed, the absolute value of conditional log-density $q_t(\theta)$ is bounded by an affine function of $|X_t|$ in the Laplacian case while it is bounded by a quadratic polynomial of $X_t$ in the Gaussian case. As a consequence, $r=1$ (respectively $r=2$) could be required for the strong consistency (resp. asymptotic normality) of the Laplacian QMLE while $r=2$ (resp. $r=4$) is required for the Gaussian QMLE. This gain on moment condition can be crucial for instance in an econometric framework where the Kurtosis of data is sometimes infinite.    
\item The proof of Theorem \ref{SC} is simpler and sharper than the proof of strong consistency in \cite{BW}. Indeed, in our new proof, we use a condition of almost sure uniform consistency based on a general and powerful result established in \cite{KW} while a Feller-type condition was "only" used in \cite{BW}. This difference leads to a very sharp condition on the decreasing rate of the Lipshitzian coefficients $(\alpha^{(0)}_k)$ for Laplacian QMLE, $\ell>1$ in \eqref{decSC}, while $\ell >3/2$ is required for Gaussian QMLE.
\item The proof of Theorem \ref{AN} is totally different to the one for Gaussian QMLE since the conditional log-density is no more differentiable with respect to the parameters. A kind of proof similar to the one used for establishing the asymptotic normality of the median is required. Hence, in a first step we had to prove an extension of a central limit for adapted processes established in \cite{DD}, {\it i.e.} our Theorem \ref{Davis}, and  we used it in a second step for establishing the asymptotic normality of the Laplacian QMLE. Note also that the conditions on the derivatives of functions $f_\theta$ and $M_\theta$ are clearly weaker with Laplacian than with Gaussian QMLE.
\end{enumerate}
\item The third related paper is \cite{FLZ}. The framework of this paper is restricted to linear causal models ($X_t=\sigma_t(\theta) \, \xi_t$) in contrast with the affine causal models ($X_t=M_\theta^t \, \xi_t+f_\theta^t$) considered in \eqref{eq::sys}. Hence ARMA but also ARMA-GARCH or ARMA-APARCH processes are not considered in this framework. Moreover the required moment is $r=4$ (instead of $r=2$ in our conditions) and the condition on the approximation of $\sigma_t(\theta)$, {\it i.e.} $\sup_\theta |\sigma_t(\theta)-\hat \sigma_t(\theta)|\leq C_1 \, \rho^t$ is clearly weaker than our Lipshitzian condition (for instance $ARCH(\infty)$ processes with Riemanian decay of the coefficients could satisfy our conditions but not their conditions). In \cite{FLZ}, a large family of instrumental probability densities, {\it i.e.} generalized Gaussian densities, including Laplace density, but their proof of asymptotic normality mimics the proof using derivatives of Gaussian QMLE since the "shift" component $f_\theta^t$ typically present for ARMA processes, is not considered in their models. Note also that \cite{FZ2015} also studies non-Gaussian QMLE but their assumption A9 implies that the Laplace density is not considered in their asymptotic normality of the QMLE. 
\end{itemize}
Finally it appears that our results provide an original extension or counterpart of these three related references.
\subsection{Examples}
In this section, several examples of time series satisfying the conditions of previous results are considered. Like it could be boring to state the results for all sufficiently famous processes, we refer, {\it mutatis mutandis}, to \cite{BW} and Bardet  {\it et al.} (2012) for ARCH$(\infty)$ and TARCH$(\infty)$.\\
~\\
{\bf 1/ APARCH processes.} APARCH$(\delta,p,q)$ model has been introduced (see \cite{DGE}) as the solution of equations
\begin{equation}\label{aparch}
\begin{cases}
& X_t=\sigma_t \zeta_t,\\
&\sigma_t^\delta=\omega+\sum_{i=1}^p \alpha_i(|X_{t-i}|-\gamma_i X_{t-i})^\delta+\sum_{j=1}^q\beta_j\sigma_{t-j}^\delta,
\end{cases}
\end{equation}
where $\delta \geq  1$, $\omega>0$, $-1 <\gamma_i < 1$ and $\alpha_i\geq0$ for $i=1,\ldots,p$, $\beta_j\geq0$ for $j=1,\ldots,q$ with $\alpha_p, \beta_q$ strictly positive and $\sum_{j=1}^q\beta_j<1$. Hence, we denote here $\theta=\big (\delta, \omega,\alpha_1,\ldots,\alpha_p, \gamma_1,\ldots,\gamma_p,\beta_1,\ldots,\beta_q  \big)$. \\
Using $L$ the usual backward operator such as $LX_t=X_{t-1}$, $\big (1-\sum_{j=1}^q\beta_j L^j \big )^{-1}$ exists and simple computations imply for $t \in \Z$:
\begin{eqnarray*}
  \sigma_t^\delta &=& \big (1-\sum_{j=1}^q\beta_j L^j \big )^{-1} \Big[ \omega+\sum_{i=1}^p \alpha_i(1-\gamma_i )^\delta (\max(X_{t-i},0))^\delta + \alpha_i(1+\gamma_i )^\delta (-\min(X_{t-i},0))^\delta \Big]     \\
    &=&  b_0+\sum_{i\geq1} b_i^+ (\max(X_{t-i},0))^\delta +\sum_{i\geq1} b_i^- (\max(-X_{t-i},0))^\delta.
\end{eqnarray*}
where $b_0={w}{(1-\sum_{j=1}^q\beta_j )^{-1}}$ and the coefficients $(b_i^+, b_i^-)_{i\geq1}$ are defined by the recursion relations
\begin{equation}\label{bb}
\begin{cases}
b_i^+=\sum_{k=1}^q \beta_k b_{i-k}^+ + \alpha_i (1-\gamma_i)^\delta   \   \  \ \mbox{with $\ \alpha_i (1-\gamma_i)=0 $ for  \ $i>p$}\\
b_i^-=\sum_{k=1}^q \beta_k b_{i-k}^- + \alpha_i (1+\gamma_i)^\delta   \   \  \ \mbox{with $\ \alpha_i (1+\gamma_i)=0$  for  \ $i>p$}
\end{cases}
\end{equation}
with $b^+_i=b_i^-=0$ for $i \leq 0$. As a consequence, for APARCH model, $f_\theta^t\equiv0$ and $M_\theta^t=\sigma_t$. It is clear that $\alpha_j^{(0)}(f,\Theta)=0$ and simple computations imply $\alpha_j^{(0)}(M,\Theta)=\sup_{\theta\in \Theta}\max\big(|b_j^+(\theta)|^{1/\delta},|b_j^-(\theta)|^{1/\delta}\big)$. Therefore $A_0(f,\Theta)$ holds and $\sum_{j=1}^q\beta_j<1$ implies that a sequence defined by $u_n=\sum_{k=1}^q \beta_k u_{n-k}$ for $n$ large enough is such as $(u_n)_{n\in \N}$ is an exponentially decreasing sequence and therefore $A_0(M,\Theta)$ holds. Thus for $r\geq 1$, the stationarity set $\Theta(r)$ is defined by
\begin{equation}\label{thetaparch}
\Theta(r)=\Big \{\theta\in\R^{2p+q+2}~\Big/~ \big ( \E \big [ |\zeta_0|^r \big
] \big )^{1/r}\sum_{j=1}^\infty \max \big (|b_j^+|^{1/\delta},|b_j^-|^{1/\delta}\big )< 1\Big \}.
\end{equation}
Now the strong consistency and asymptotic normality of the Laplacian-QMLE for APARCH models can be established (see the proof in Section \ref{Proofs}):
\begin{proposition}\label{QMLAPARCH}
Assume that $X$ is a stationary solution of \eqref{aparch} with $\theta_0 \in \Theta$ where $\Theta$ is a compact subset of $\Theta(r)$ defined in \eqref{thetaparch}. Then,
\begin{enumerate}
  \item If $r=1$, then $\widehat{\theta}_n   \limitepsn \theta_0$.
  \item If $r=2$, and if $\Gamma_M$ defined in \eqref{Gamma} is a definite positive symmetric matrix, then
\begin{eqnarray*}
\sqrt{n}\big (\widehat\theta_n-\theta_0\big ) \limiteloin {\cal N}_{2p+q+2}\big (0 \ , \ ( \sigma^2_\zeta -1  ) \  \Gamma_M^{-1}  \big ).
\end{eqnarray*}
\end{enumerate}
\end{proposition}
\noindent To our knowledge, this is the first statement the asymptotic properties of Laplacian-QMLE for APARCH processes.\\
~\\
{\bf 2/ ARMA-GARCH processes.}  ARMA$(p,q)$-GARCH$(p',q')$ processes have been
introduced by \cite{DGE} and \cite{LM} as the solution of the system of equations
\begin{equation}\label{armaga}
\begin{cases}
&P_\theta(L) \ X_t=Q_\theta(L) \ \varepsilon_t,\\
&\varepsilon_t=\sigma_t \zeta_t,\quad \mbox{with} \quad \sigma_t^2=c_0+\sum_{i=1}^{p'}c_i \varepsilon_{t-i}^2+\sum_{i=1}^{q'}d_i \sigma_{t-i}^2
\end{cases}
\end{equation}
where
\begin{itemize}
\item $c_0>0$, $c_i\geq0$ for $i=1,\ldots,p'$, $d_i\geq0$ for $i=1,\ldots,q'$, $\sum_{i=1}^{q'}d_i<1$ and $c_{p'}, d_{q'}$ positive;
\item  $P_\theta(x)=1-a_1x-\cdots-a_px^p$ and $Q_\theta(x)=1-b_1x-\cdots-b_{q}x^{q}$ are coprime polynomials with $\sum_{i=1}^p |a_i|<1$ and $\sum_{i=1}^p |b_i|<1$.
\end{itemize}
Let $\theta=(c_0, c_1,\ldots,c_{p'},d_1,\ldots,d_{q'},a_1,\ldots, a_p, b_1,\ldots,b_{q} )$.
We are going to use Lemma \ref{cond2}. Since $(\varepsilon_t)$ is supposed to be a GARCH$(p',q')$, then $f^\varepsilon_\theta=0$ and $M^\varepsilon_\theta=\big (\big (1-\sum_{j=1}^{q'}d _j L^j \big )^{-1}\big (c_0+c_1 \varepsilon^2_{t-1}+ \cdots + c_{p'} \varepsilon^2_{t-p'} \big )^{1/2}$ and direct computations imply that the Lipshitz coefficients of $(\varepsilon_t)$ are such as $\alpha_j^{(0)}(f^\varepsilon,\{\theta_0\} )=0$ and $\alpha_j^{(0)}(M^\varepsilon,\{\theta_0\} )=|\beta_j|$ with $\big (1+\sum_{j=1}^\infty \beta_j x^j\big )\big (1-\sum_{j=1}^{q'}d _j x^j \big )=\sum_{j=0}^{p'} c_j x^j$. Therefore $(\alpha_j^{(0)}(f^\varepsilon,\{\theta_0\} ))_j$ and $(\alpha_j^{(0)}(M^\varepsilon,\{\theta_0\} ))_j$ are EDS (see for instance \cite{BH}). Thus $(A_0(f^\varepsilon,\{\theta_0\}))$ and $(A_0(M^\varepsilon,\{\theta_0\}))$ hold. \\
Considering the ARMA part and denoting $(\psi_j)$ such as $\big ( 1 +\sum_{j=1}^\infty \psi_j x^j\big ) \big ( 1-\sum_{j=1}^\infty a_j x^j\big )=  \big ( 1-\sum_{j=1}^\infty b_j x^j\big )$, then  from Lemma \ref{cond2} we deduce that:
$$
\left \{ \begin{array}{lcl}
\alpha_j^{(0)}(f,\{\theta_0\} ) &=& |\psi_j | \\
\alpha_j^{(0)}(M,\{\theta_0\} ) &\leq & \sum_{k=1}^j |\psi_k |\times  |\beta_{j-k} |
\end{array} . \right .
$$
Then we deduce that $(\alpha_j^{(0)}(f,\{\theta_0\} ))_j$ and $(\alpha_j^{(0)}(M,\{\theta_0\} ))_j$ are EDS, $(A_0(f,\{\theta_0\}))$ and $(A_0(M,\{\theta_0\}))$ hold, and $X$ is a.s. a solution of (\ref{eq::sys}) for $\theta$ included in the $r$-order stationarity set $\Theta(r)$ defined by
\begin{equation}\label{thetaAG}
\Theta(r)=\Big \{\theta\in\R^{p+q+p'+q'+1}~\Big/~
\sum_{i=1}^\infty|\psi_i(\theta)|+\big ( \E \big [ | \zeta_0|^r \big
] \big )^{1/r}\sum_{j=1}^\infty \sum_{k=1}^j |\psi_k |\times  |\beta_{j-k} |< 1\Big \}.
\end{equation}
Now the strong consistency and asymptotic normality of the Laplacian-QMLE for ARMA-GARCH processes can be established:
\begin{proposition}\label{QMLArmagarch}
Assume that $X$ is a stationary solution of \eqref{armaga} with $\theta_0 \in \Theta$ where $\Theta$ is a compact subset of $\Theta(r)$ defined in \eqref{thetaAG}. Then,
\begin{enumerate}
  \item If $r=1$, then $\widehat{\theta}_n   \limitepsn \theta_0$.
  \item If $r=2$, and if $\Gamma_f$ and $\Gamma_M$ defined in \eqref{Gamma} are definite positive symmetric matrix, then the asymptotic normality \eqref{tlcqmle} of $\widehat{\theta}_n$ holds.
\end{enumerate}
\end{proposition}
\noindent This result is a new one and extends the previous results already obtained with Gaussian-QMLE for such processes (see for instance, \cite{LM} and \cite{BW}). \\
~\\
~\\
{\bf 3/ ARMA-ARCH$(\infty)$ processes.}  ARMA$(p,q)$-ARCH$(\infty)$ processes are a natural extension of ARMA-GARCH processes. They are the solution of the system of equations
\begin{equation}\label{armaa}
\begin{cases}
&P_\theta(L) \ X_t=Q_\theta(L) \ \varepsilon_t,\\
&\varepsilon_t=\sigma_t \zeta_t,\quad \mbox{with} \quad \sigma_t^2=c_0+\sum_{i=1}^{\infty}c_i \varepsilon_{t-i}^2
\end{cases}
\end{equation}
where
\begin{itemize}
\item $c_0>0$, $c_i\geq0$ for $i\geq 1$;
\item  $P_\theta(x)=1-a_1x-\cdots-a_px^p$ and $Q_\theta(x)=1-b_1x-\cdots-b_{q}x^{q}$ are coprime polynomials with $\sum_{i=1}^p |a_i|<1$ and $\sum_{i=1}^p |b_i|<1$.
\end{itemize}
ARCH$(\infty)$ processes were introduced by \cite{R} and the asymptotic properties of Gaussian-QMLE were studied in \cite{RZ}, \cite{SM} or \cite{BW}.
Hence, we assume that there exists $\beta=(\beta_1,\cdots,\beta_m)$ such as for all $i\in \N$, $c_i=c(i,\beta)$, with $c(\cdot)$ a known function.
Let $\theta=(\beta,a_1,\ldots, a_p, b_1,\ldots,b_{q} )$.
We are going to use Lemma \ref{cond2}. Since $(\varepsilon_t)$ is supposed to be an ARCH$(\infty)$, then $f^\varepsilon_\theta=0$ and $M^\varepsilon_\theta=\big (c(0,\beta)+\sum_{i=1}^\infty c(i,\beta) \varepsilon^2_{t-i} \big )^{1/2}$ and direct computations imply that the Lipshitz coefficients of $(\varepsilon_t)$ are such as $\alpha_j^{(0)}(f^\varepsilon,\{\theta_0\} )=0$ and $\alpha_j^{(0)}(M^\varepsilon,\{\theta_0\} )=c(j,\beta_0)$. Therefore we assume that there exists $\ell > 1$ such as
\begin{eqnarray} \label{condAA}
c(j,\beta_0)= {\cal O} \big ( j^{-\ell})~\mbox{when}~j \to  \infty.
\end{eqnarray}
Thus $(A_0(f^\varepsilon,\{\theta_0\}))$ and $(A_0(M^\varepsilon,\{\theta_0\}))$ hold. \\
Considering the ARMA part and denoting $(\psi_j)$ such as $\big ( 1 +\sum_{j=1}^\infty \psi_j x^j\big ) \big ( 1-\sum_{j=1}^\infty a_j x^j\big )=  \big ( 1-\sum_{j=1}^\infty b_j x^j\big )$, then  from Lemma \ref{cond2} we deduce that:
$$
\left \{ \begin{array}{lcl}
\alpha_j^{(0)}(f,\{\theta_0\} ) &=& |\psi_j | \\
\alpha_j^{(0)}(M,\{\theta_0\} ) &\leq & \sum_{k=1}^j |\psi_k |\times  c(j,\beta_0)
\end{array} . \right .
$$
Then we deduce that $(\alpha_j^{(0)}(f,\{\theta_0\} ))_j$ is EDS and $(\alpha_j^{(0)}(M,\{\theta_0\} ))_j= {\cal O} \big ( j^{-\ell})$. Then $(A_0(f,\{\theta_0\}))$ and $(A_0(M,\{\theta_0\}))$ hold, and $X$ is a.s. a solution of (\ref{eq::sys}) for $\theta$ included in the $r$-order stationarity set $\Theta(r)$ defined by
\begin{equation}\label{thetaAA}
\Theta(r)=\Big \{\theta\in\R^{p+q+m}~\Big/~
\sum_{i=1}^\infty|\psi_i(\theta)|+\big ( \E \big [ | \zeta_0|^r \big
] \big )^{1/r}\sum_{j=1}^\infty \sum_{k=1}^j |\psi_k |\times c(j,\beta_0)< 1\Big \}.
\end{equation}
Now the strong consistency and asymptotic normality of the Laplacian-QMLE for ARMA-ARCH$(\infty)$ processes can be established:
\begin{proposition}\label{QMLArmaarch}
Assume that $X$ is a stationary solution of \eqref{armaa} where \eqref{condAA} holds and with $\theta_0 \in \Theta$ where $\Theta$ is a compact subset of $\Theta(r)$ defined in \eqref{thetaAA}. Then,
\begin{enumerate}
  \item If $r \geq 1$ and $\ell \geq 2/\min(r,2)$, then $\widehat{\theta}_n   \limitepsn \theta_0$.
  \item If $r=2$, $\ell>1$ and if $\partial^i _\beta c(j,\beta)={\cal O}\big (j^{-\ell} \big )$ for $i=1,~2$, and if $\Gamma_f$ and $\Gamma_M$ defined in \eqref{Gamma} are definite positive symmetric matrix, then the asymptotic normality \eqref{tlcqmle} of $\widehat{\theta}_n$ holds.
\end{enumerate}
\end{proposition}
\noindent This result is a new one. Note that $\ell>1$ and $r=2$ is required for the asymptotic normality of Laplacian-QMLE while $r=4$ and $\ell>2$ is required for Gaussian-QMLE for such processes (see for instance \cite{BW}). This confers a clear advantage to Laplacian-QMLE. \\
~\\
{\bf 4/ ARMA-APARCH processes.}
The ARMA$(p,q)$-APARCH$(p',q')$ processes have been also introduced by \cite{DGE} as the solutions of the equations
\begin{equation}\label{armaap}
\begin{cases}
&P_\theta(L) \, X_t=Q_\theta(L) \, \varepsilon_t,\\
&\varepsilon_t=\sigma_t \, \zeta_t,~\mbox{with} ~\sigma_t^\delta=\omega+\sum_{i=1}^{p'} \alpha_i(|\varepsilon_{t-i}|-\gamma_i \varepsilon_{t-i})^\delta+\sum_{j=1}^{q'}\beta_j\sigma_{t-j}^\delta
\end{cases}
\end{equation}
where:
\begin{itemize}
\item  $\delta \geq 1$, $\omega>0$, $-1 <\gamma_i < 1$ and $\alpha_i\geq 0$ for $i=1,\ldots,p'-1$, $\beta_j\geq0$ for $j=1,\ldots,q'-1$, $\alpha_{p'}, \,\beta_{q'}$ positive real numbers and $\sum_{j=1}^{p'} \alpha_j<1$;
\item $P_\theta(x)=1-a_1x-\cdots-a_px^p$ and $\Psi_\theta(x)=1-b_1x-\cdots-b_{q}x^{q}$ are coprime polynomials with $\sum_{i=1}^p |a_i|<1$ and $\sum_{i=1}^q |b_i|<1$ .
\end{itemize}
Let $\theta=(\delta,\omega,\alpha_1,\ldots,\alpha_{p'},\gamma_1,\ldots,\gamma_{p'}, \beta_1,\ldots,\beta_{q'},a_1,\ldots, a_p, b_1,\ldots,b_{q} )$.  Then, as for ARMA-GARCH processes, we are going to use Lemma \ref{cond2}. Thanks to the computations realized for APARCH processes, we obtain $\alpha_j^{(0)}(f^\varepsilon,\{\theta_0\} )=0$ and $\alpha_j^{(0)}(M^\varepsilon,\{\theta_0\} )=\max (|b^+_j|^{1/\delta}\, , \, |b^-_j|^{1/\delta})$ with  $(b_i^+,b_i^-)_{i\geq1}$ defined in \eqref{bb}. \\
Then, we have
$$
\left \{ \begin{array}{lcl}
\alpha_j^{(0)}(f,\{\theta_0\} ) &\leq& |\psi_j | \\
\alpha_j^{(0)}(M,\{\theta_0\} ) &\leq & \sum_{k=1}^j |\psi_k |\times  \max \big (|b^+_{j-k}|^{1/\delta}\, , \, |b^-_{j-k}|^{1/\delta} \big )
\end{array} . \right .
$$
$(\psi_j)$ such as $\big ( 1 +\sum_{j=1}^\infty \psi_j x^j\big ) \big ( 1-\sum_{j=1}^\infty a_j x^j\big )=  \big ( 1-\sum_{j=1}^\infty b_j x^j\big )$. From Lemma \ref{cond2}, $(A_0(f,\Theta))$  and $(A_0(M,\Theta))$ hold since $(\alpha_j^{(0)}(f^\varepsilon,\{\theta_0\} ))_j=0$ and $(\alpha_j^{(0)}(M^\varepsilon,\{\theta_0\} ))_j$ are EDS. As a consequence, for $r\geq 1$, the stationarity set $\Theta(r)$ is defined by
$$
\Theta(r)=\Big\{\theta\in\R^{p+q+p'+q'+2}~\Big/~ \sum_{j=1}^\infty|\psi_j| +\big ( \E \big | \zeta_0|^r \big
] \big )^{1/r}\sum_{j=1}^\infty \sum_{k=1}^j |\psi_k |\times  \max \big (|b^+_{j-k}|^{1/\delta}\, , \, |b^-_{j-k}|^{1/\delta} \big )< 1\Big \}.
$$
Now, we are able to provide the asymptotic properties of QMLE for ARMA-APARCH models.
\begin{proposition}\label{QMLArmaparch}
Assume that $X$ is a stationary solution of \eqref{armaap} with $\theta_0 \in \Theta$ where $\Theta$ is a compact subset of $\Theta(r)$ defined in \eqref{thetaAG}. Then,
\begin{enumerate}
  \item If $r=1$, then $\widehat{\theta}_n   \limitepsn \theta_0$.
  \item If $r=2$, and if $\Gamma_f$ and $\Gamma_M$ defined in \eqref{Gamma} are definite positive symmetric matrix, then the asymptotic normality \eqref{tlcqmle} of $\widehat{\theta}_n$ holds.
\end{enumerate}
\end{proposition}
\noindent This result is stated for the first time for Laplacian-QMLE. The case of Gaussian-QMLE for ARMA-APARCH could be also obtained following the previous decomposition and the paper \cite{BW}. Once again, the asymptotic normality of Laplacian-QMLE only requires $r=2$ while this requires $r=4$ for Gaussian-QMLE.
\section{Numerical Results} \label{Simu}
To illustrate the asymptotic results stated previously, we realized Monte-Carlo experiments on the bevarior of Laplacian-QMLE (denoted $\widehat \theta_n^{LQL}$) for several time series models, sample sizes  and probability distributions. A comparison with the results obtained by Gaussian QMLE (denoted $\widehat \theta_n^{GQL}$) is also proposed. %
~\\ 
More precisely, the considered probability distributions of $(\zeta_t)$ are:
\begin{itemize}
\item Centred Gaussian distribution denoted ${\cal N}$;
\item Centred Laplacian distribution denoted ${\cal L}$;
\item Centred Uniform distribution denoted ${\cal U}$;
\item Centred Student distribution with $3$ freedom degrees, denoted $t_3$;
\item Normalized centred Gaussian mixture with probability distribution $0.05*{\cal N}(-2,0.16)+{\cal N}(0,1)+0.05*{\cal N}(2,0.16)$ and denoted ${\cal M}$.
\end{itemize}
All these probability distributions are normalized such as $\E [|\zeta_0|]=1$, required for Laplacian-QMLE. For using Gaussian-QMLE requiring $\sigma^2_\zeta=1$, it is necessary to consider the model with $M'_\theta=\frac {\E [|\zeta_0|]} {\sigma_\zeta} \, M_\theta$ instead of $M_\theta$.\\
Several models of time series satisfying \eqref{eq::sys} and the assumptions of Theorem \ref{SC} and \ref{AN} are considered:
\begin{itemize}
\item a ARMA$(1,1)$ process defined by $X_t=\phi \  X_{t-1} +  \zeta_t + \theta \zeta_{t-1}$ \  \  with $\phi=0.4 $ and $\theta=0.6$;
  \item a ARCH$(1)$ process defined by $X_t=\zeta_t \, \sqrt{\omega+\alpha X_{t-1}^2} $ \  \  with $\omega=0.4$ and $\alpha= 0.2$;
  \item a GARCH$(1,1)$ process defined by  $X_t=\zeta_t \, \sigma_t$ where $\sigma^2_t=\alpha_0+\alpha_1 X_{t-1}^2+\beta\sigma_{t-1}^2 $ with $\alpha_0=0.2$, $\alpha_1= 0.4$ and $\beta=0.2$;
  \item a ARMA$(1,1)$-GARCH$(1,1)$ process defined by $X_t=\phi X_{t-1}+\varepsilon_t+\theta \varepsilon_{t-1}$ where $\varepsilon_t=\zeta_t \,\sigma_t$ and $\sigma_t^2= \alpha_0+\alpha_1 \varepsilon_{t-1}^2+\beta\sigma_{t-1}^2 $ with $\phi=0.4,  \, \theta=0.6, \, \alpha_0=0.2, \,  \alpha_1=0.4$ and $\beta=0.1$;
  \item a ARMA$(1,1)$-APARCH$(1,1)$ process defined by  $X_t=\phi X_{t-1}+\varepsilon_t+\theta \varepsilon_{t-1}$
   where $\varepsilon_t=\zeta_t \,\sigma_t$ and $\sigma_t^\delta = \alpha_0+\alpha_1 \big ( |\varepsilon_{t-1}|-\gamma \varepsilon_{t-1} \big ) ^\delta +\beta\sigma_{t-1}^\delta $ and $\phi=0.4$, $\theta=0.6$, $\alpha_0=0.2$, $\alpha_1=0.4$,  $\gamma=0.5$, $\beta=0.1$ and  $\delta=1.2$.
\end{itemize}
Hence we computed the root-mean-square error (RMSE) from $1000$ independent replications of $\widehat \theta_n^{LQL}$ and $\widehat \theta_n^{LQL}$ for those processes and the results are presented in Table \ref{Table1} on
page~\pageref{Table1} and \ref{Table2} on
page~\pageref{Table2}.\\

\begin{table*}
\caption{Root Mean Square Error of the components of $\widehat{\theta}_n^{LQL}$ and $\widehat{\theta}_n^{GQL}$ for ARMA$(1,1)$, ARCH$(1)$ and GARCH$(1,1)$ processes.}
\label{Table1}
\centering
\begin{tabular}{lll|c|c|c|c|c|c|c|c|c|c|}
\cline{4-4}\cline{5-5}\cline{6-6}\cline{7-7}\cline{8-8}\cline{9-9}\cline{10-10}\cline{11-11}\cline{12-12}\cline{13-13}
&  &  & \multicolumn{2}{c|}{${\cal L}$} & \multicolumn{2}{c|}{${\cal N}$} & \multicolumn{2}{c|}{${t_3}$} &
\multicolumn{2}{c|}{${\cal U}$} & \multicolumn{2}{c|}{${\cal M}$} \\
\cline{3-3}\cline{4-4}\cline{5-5}\cline{6-6}\cline{7-7}\cline{8-8}\cline{9-9}\cline{10-10}\cline{11-11}\cline{12-12}\cline{13-13}
&  & $n$ & $\widehat \theta_n^{GQL}$ & $\widehat \theta_n^{LQL}$ & $\widehat \theta_n^{GQL}$ & $\widehat \theta_n^{LQL}$ & $\widehat \theta_n^{GQL}$ & $\widehat \theta_n^{LQL}$ & $\widehat \theta_n^{GQL}$ & $\widehat \theta_n^{LQL}$ & $\widehat \theta_n^{GQL}$ & $\widehat \theta_n^{LQL}$ \\
\cline{1-1}\cline{2-2}\cline{3-3}\cline{4-4}\cline{5-5}\cline{6-6}\cline{7-7}\cline{8-8}\cline{9-9}\cline{10-10}\cline{11-11}\cline{12-12}\cline{13-13}
ARMA(1,1) & $\theta$ & 100 &   0.106 &0.091& 0.114   & 0.117  & 0.113  & 0.090   &0.112 & 0.059  &0.110   &  0.078   \\
&  & 1000 &  0.031  & 0.024  & 0.032 & 0.032  & 0.036  & 0.027   &0.031 & 0.014  &  0.031 & 0.023    \\
&  & $5000$ &  0.014  &0.010   &  0.014  & 0.015  & 0.016  &0.011    &0.016 & 0.012  &0.013   &  0.010   \\
\cline{2-2}\cline{3-3}\cline{4-4}\cline{5-5}\cline{6-6}\cline{7-7}\cline{8-8}\cline{9-9}\cline{10-10}\cline{11-11}\cline{12-12}\cline{13-13}
& $\phi$ & 100  &0.119& 0.102  &  0.121  & 0.128  &0.123   &0.102&0.120 & 0.067  &0.121   &0.090     \\
&  & 1000 &  0.037  & 0.028  & 0.036   & 0.036  &0.040   & 0.030   &0.036 & 0.017  & 0.036  & 0.027    \\
&  & $5000$ &   0.016 &0.012   & 0.014   & 0.016  & 0.017  & 0.013   &0.016 &  0.007 &  0.014 &0.006     \\
\hline
\hline
ARCH(1) & $\omega$ & 100 &  0.068 & 0.061  &  0.048 & 0.049 & 0.254 &0.085   & 0.035 &  0.025  &  0.062&  0.052   \\
&  & 1000 & 0.020  & 0.018  & 0.015  & 0.015 &0.134 &0.049 & 0.011& 0.016   & 0.036 &0.018   \\
&  & $5000$ & 0.010  & 0.009  & 0.006  &  0.006&0.115  &0.044   &0.005 &  0.015   &0.031  & 0.008  \\
\cline{2-2}\cline{3-3}\cline{4-4}\cline{5-5}\cline{6-6}\cline{7-7}\cline{8-8}\cline{9-9}\cline{10-10}\cline{11-11}\cline{12-12}\cline{13-13}
& $\alpha$ & 100  &0.161   &0.155   & 0.141  &0.142    & 0.979& 0.418 & 0.102& 0.064 & 0.484& 0.423   \\
&  & 1000 &  0.063 &  0.058 &  0.043 &  0.043&0.852 &0.169 &0.029& 0.033   &  0.157& 0.133  \\
&  & $5000$ &  0.016 &  0.014 &  0.012 & 0.012 &    0.378 & 0.109 & 0.013  &   0.031     &  0.087& 0.062  \\
\hline
\hline
GARCH(1,1)&  $\alpha_0$&100 & 0.112  & 0.105  & 0.095  &0.100  & 0.211  &0.126    &0.081 &0.047 & 0.134 &0.114   \\
&   & 1000 &  0.036 &0.032   & 0.028  &0.028   & 0.098 & 0.058 &0.023   &0.018   &0.066  &0.051   \\
 & & $5000$ & 0.016  & 0.014  &0.012   &0.012   & 0.055&0.043 &0.010   &0.015   &  0.040& 0.023  \\
 \cline{2-2}\cline{3-3}\cline{4-4}\cline{5-5}\cline{6-6}\cline{7-7}\cline{8-8}\cline{9-9}\cline{10-10}\cline{11-11}\cline{12-12}\cline{13-13}
& $\alpha_1$ & 100 & 0.162  &0.157   & 0.149  & 0.150  &  0.453&0.364  & 0.115 & 0.070& 0.507   &0.429    \\
 & & 1000 &  0.061 & 0.056  &0.449   &  0.449 &0.333 &0.150   & 0.030 &0.033&0.160   &0.136   \\
 & & $5000$ &  0.029 & 0.026  &  0.020 & 0.020  &0.193  &0.095  & 0.013  &0.030   &0.086  &0.058   \\
\cline{2-2}\cline{3-3}\cline{4-4}\cline{5-5}\cline{6-6}\cline{7-7}\cline{8-8}\cline{9-9}\cline{10-10}\cline{11-11} \cline{12-12}\cline{13-13}
& $\beta$ & 100 &  0.225 &0.209   &  0.188 & 0.190  & 0.499& 0.429    &0.163&0.105& 0.483  &0.390    \\
&  & 1000 &  0.060 & 0.055  &0.051   &0.051   & 0.285 & 0.174  &0.044&0.022& 0.170  & 0.169     \\
&  & $5000$ & 0.027  &0.024   & 0.022  & 0.022  & 0.180 &  0.075  &  0.019 &0.009   & 0.072 &0.075   \\
\hline \\
\end{tabular}
\end{table*}

\begin{table*}
\caption{Root Mean Square Error of the components of $\widehat{\theta}_n^{LQL}$ and $\widehat{\theta}_n^{GQL}$ for ARMA$(1,1)$-GARCH$(1,1)$ and ARMA$(1,1)$-APARCH$(1,1)$ processes.}
\label{Table2}
\centering
\begin{tabular}{lll|c|c|c|c|c|c|c|c|c|c|}
\cline{4-4}\cline{5-5}\cline{6-6}\cline{7-7}\cline{8-8}\cline{9-9}\cline{10-10}\cline{11-11}\cline{12-12}\cline{13-13}
&  &  & \multicolumn{2}{c|}{${\cal L}$} & \multicolumn{2}{c|}{${\cal N}$} & \multicolumn{2}{c|}{${t_3}$} &
\multicolumn{2}{c|}{${\cal U}$} & \multicolumn{2}{c|}{${\cal M}$} \\
\cline{3-3}\cline{4-4}\cline{5-5}\cline{6-6}\cline{7-7}\cline{8-8}\cline{9-9}\cline{10-10}\cline{11-11}\cline{12-12}\cline{13-13}
&  & $n$ & $\widehat \theta_n^{GQL}$ & $\widehat \theta_n^{LQL}$ & $\widehat \theta_n^{GQL}$ & $\widehat \theta_n^{LQL}$ & $\widehat \theta_n^{GQL}$ & $\widehat \theta_n^{LQL}$ & $\widehat \theta_n^{GQL}$ & $\widehat \theta_n^{LQL}$ & $\widehat \theta_n^{GQL}$ & $\widehat \theta_n^{LQL}$ \\
\cline{1-1}\cline{2-2}\cline{3-3}\cline{4-4}\cline{5-5}\cline{6-6}\cline{7-7}\cline{8-8}\cline{9-9}\cline{10-10}\cline{11-11}\cline{12-12}\cline{13-13}
\hline
ARMA$(1,1)$  & $\theta$ & 100 & 0.120 &0.097  &0.107  &0.107   &0.121 &0.098   & 0.097  &0.067  &0.123   &0.087   \\
-GARCH$(1,1)$ & & 1000 &0.035  &0.024  &0.028  &0.028   & 0.048 & 0.029  & 0.024  & 0.015 &0.035   &0.026   \\
&  & $5000$ &  0.016& 0.010 &0.012  &0.012& 0.023 &0.012  &0.015  &0.011   &0.011&0.007\\
\cline{2-2}\cline{3-3}\cline{4-4}\cline{5-5}\cline{6-6}\cline{7-7}\cline{8-8}\cline{9-9}\cline{10-10}\cline{11-11}\cline{12-12}\cline{13-13}
& $\phi$ & 100 &0.135  & 0.109&0.117 & 0.119 &0.141 & 0.116  &0.110   &0.077 & 0.132  &0.102   \\
&  & 1000 &0.044  &0.030  &0.033  &0.033   & 0.063 &0.035  &  0.029 &0.023  &0.053   &0.046   \\
&  & $5000$ &0.020  &0.014  &0.015  &0.015& 0.029 & 0.015 & 0.013&0.012 &0.019  &0.014   \\
\cline{2-2}\cline{3-3}\cline{4-4}\cline{5-5}\cline{6-6}\cline{7-7}\cline{8-8}\cline{9-9}\cline{10-10}\cline{11-11}\cline{12-12}\cline{13-13}
& $\alpha_0$&100& 0.104 & 0.096 &0.085  &0.084   &0.158 & 0.129 &  0.073 & 0.055 &0.131   &0.116   \\
&  & 1000 & 0.031 & 0.028 &0.025  &0.025   & 0.241 & 0.060& 0.021  &0.019  & 0.053  &0.046   \\
&  & $5000$& 0.014  &0.012&0.010  &0.010  & 0.052 & 0.042  &0.009   &0.016 &0.036&0.019  \\
\cline{2-2}\cline{3-3}\cline{4-4}\cline{5-5}\cline{6-6}\cline{7-7}\cline{8-8}\cline{9-9}\cline{10-10}\cline{11-11}\cline{12-12}\cline{13-13}
& $\alpha_1$ &100& 0.179 &0.177  &0.166  &0.167   &0.469 & 0.385  & 0.134  &0.107  &0.494   &0.405   \\
&  & 1000 & 0.064& 0.060 &0.045  &0.045   & 0.328 &0.161 & 0.031  &0.046&0.160     &0.137   \\
&  & $5000$ &0.031  & 0.027 & 0.020 &0.020   &0.182  &0.096  &0.013   & 0.038 &0.090   &0.062   \\
\cline{2-2}\cline{3-3}\cline{4-4}\cline{5-5}\cline{6-6}\cline{7-7}\cline{8-8}\cline{9-9}\cline{10-10}\cline{11-11}\cline{12-12}\cline{13-13}
& $\beta$ & 100 & 0.302 & 0.269 &0.252  &0.233   & 0.604 & 0.497 &0.217   &0.164  &0.553   &0.472   \\
&  & 1000& 0.057& 0.051&  0.051& 0.051  & 0.312  & 0.187  &0.045   &0.049  &0.165   &0.170   \\
&  & $5000$ & 0.025 & 0.022 & 0.020 &0.020   & 0.199 & 0.073 &0.062   &0.066  &0.019   &0.025   \\
\hline
\hline
ARMA$(1,1)$& $\theta$ & 100 & 0.110 &0.086   &0.096  &0.101   & 0.112  &0.090   & 0.097  &0.068  &0.125   & 0.091    \\
-APARCH$(1,1)$  & & 1000 &0.029  &0.021   &  0.023& 0.024   & 0.031  & 0.021  &0.022   & 0.014    &   0.033 &0.024 \\
  & &  $5000$ & 0.013 & 0.008  &0.010  &0.010   & 0.014  &0.009   &0.010   & 0.006    & 0.015 & 0.011\\
\cline{2-2}\cline{3-3}\cline{4-4}\cline{5-5}\cline{6-6}\cline{7-7}\cline{8-8}\cline{9-9}\cline{10-10}\cline{11-11}\cline{12-12}\cline{13-13}
 &$\phi$ & 100 &  0.138&  0.114 &  0.121&  0.126 & 0.128  &0.107   & 0.111  & 0.086    &  0.146  &0.107 \\
 & &   1000 & 0.040 & 0.027  & 0.032 & 0.032 & 0.041  & 0.028  &  0.029 &0.026     & 0.043  &0.030 \\
 & &   $5000$ & 0.018 & 0.012  &  0.012& 0.012  &  0.020 & 0.012  & 0.013  &0.014     &    0.019& 0.013\\
\cline{2-2}\cline{3-3}\cline{4-4}\cline{5-5}\cline{6-6}\cline{7-7}\cline{8-8}\cline{9-9}\cline{10-10}\cline{11-11}\cline{12-12}\cline{13-13}
 &$\omega$ & 100 & 0.198 & 0.192  &0.199  & 0.210  & 0.254  & 0.262  & 0.221  & 0.170    & 0.290   & 0.272\\
  & &  1000 &  0.079 & 0.067 & 0.056&0.056  & 0.226  & 0.218  &  0.044 & 0.045    &  0.142  &0.129 \\
 & &   $5000$ & 0.033 & 0.028  & 0.025 & 0.025  &  0.209 & 0.207  &0.017   & 0.029    &   0.061 & 0.056\\
\cline{2-2}\cline{3-3}\cline{4-4}\cline{5-5}\cline{6-6}\cline{7-7}\cline{8-8}\cline{9-9}\cline{10-10}\cline{11-11}\cline{12-12}\cline{13-13}
 &$\alpha$ & 100 & 0.206 & 0.201  & 0.183 & 0.184  & 0.464  &0.449   & 0.167  & 0.131    &   0.352 &0.327 \\
& &    1000 &0.060  & 0.053  &     0.041   & 0.041 & 0.447   & 0.432 &  0.029& 0.043  &   0.143 &0.134 \\
  & &  $5000$ &  0.025& 0.023  & 0.018 &0.018   &0.421   & 0.414  & 0.012  & 0.027    &  0.071  & 0.059\\
\cline{2-2}\cline{3-3}\cline{4-4}\cline{5-5}\cline{6-6}\cline{7-7}\cline{8-8}\cline{9-9}\cline{10-10}\cline{11-11}\cline{12-12}\cline{13-13}
 &$\gamma$ & 100 &  0.413& 0.386  & 0.346 &0.356   & 0.439  &0.426   & 0.310  & 0.233    & 0.613   & 0.601\\
  & &  1000 & 0.105 & 0.094  &0.071 & 0.070  & 0.101  & 0.092  & 0.057  &0.041     &  0.217  &0.220 \\
 & &   $5000$ &0.042  & 0.039  &0.029  &0.029   &  0.045 & 0.038  &   0.024& 0.018    & 0.086   & 0.089\\
\cline{2-2}\cline{3-3}\cline{4-4}\cline{5-5}\cline{6-6}\cline{7-7}\cline{8-8}\cline{9-9}\cline{10-10}\cline{11-11}\cline{12-12}\cline{13-13}
 &$\beta$ & 100 & 0.297 & 0.282  & 0.255 & 0.238  &  0.312 & 0.288  & 0.186  &0.145     &  0.476  &0.468 \\
  & &  1000 & 0.074 & 0.067  &    0.058 & 0.058 &  0.074 & 0.066  & 0.043  & 0.033    &  0.151  & 0.150\\
 & &   $5000$ & 0.033 &0.031   & 0.024 & 0.024  &  0.034 & 0.029  &  0.018 &0.012     &   0.061 & 0.063\\
\cline{2-2}\cline{3-3}\cline{4-4}\cline{5-5}\cline{6-6}\cline{7-7}\cline{8-8}\cline{9-9}\cline{10-10}\cline{11-11}\cline{12-12}\cline{13-13}
 &$\delta$ & 100 & 0.732 & 0.732  &  0.712& 0.717  &  0.740 & 0.739  & 0.703  & 0.613    &   0.830 & 0.815\\
 & &   1000& 0.402 & 0.352  &  0.296  &0.296   &   0.394  & 0.338  & 0.235&0.290    &  0.542  &0.534 \\
& &    $5000$ &  0.170& 0.147  &0.132  &0.131   &  0.169 & 0.145  &0.092   &0.168     &   0.251 & 0.262\\
\hline \\
\end{tabular}
\end{table*}
\noindent {\bf Conclusion of the numerical results:} On the one hand, it is clear that the RMSE decreases as the sample size increases, which validates the theoretical results (consistency of the estimators). On the other hand, Table \ref{Table1} and \ref{Table2} show that the Laplacian-QMLE provides more accurate estimation than the Gaussian-QMLE for several types of noise, except of course in the case of a Gaussian distribution (even in this case the RSME of both the estimators are almost the same).
\section{Proofs} \label{Proofs}
\begin{proof}[Proof of Lemma \ref{cond2}]
First, as $X$ is a stationary process and the ARMA$(p,q)$ process is causal invertible then $\widetilde X$ is also a stationary process (the coefficients of $\Lambda_{\beta_0}$ are EDS). Moreover, it is well known that $(\psi_j(\beta_0))_{j\in \N}$ is EDS. Then we have:
\begin{eqnarray}
\nonumber \widetilde X_t & =& \Lambda_{\beta_0}(L) \ \Big ( M _{\theta_0}\big (( X_{t-i})_{i\geq 1}\big ) \, \zeta_t + f _{\theta_0}\big (( X_{t-i})_{i\geq 1}\big ) \Big ) \\
\nonumber \widetilde X_t +\sum_{j=1}^\infty \psi_j(\beta_0) \widetilde X_{t-j}&=& M _{\theta_0}\big ((\Lambda_{\beta_0}^{-1}(L) \widetilde X_{t-i})_{i\geq 1}\big ) \, \zeta_t+ f _{\theta_0}\big ((\Lambda_{\beta_0}^{-1}(L) \widetilde X_{t-i})_{i\geq 1}\big ) \\
\nonumber \widetilde X_t &=& \widetilde M _{\widetilde \theta_0}\big ( (\widetilde X_{t-i})_{i\geq 1}\big ) \, \zeta_t+ \widetilde f _{\widetilde \theta_0}\big (( \widetilde X_{t-i})_{i\geq 1}\big ) \\
&& \label{Mftilde} \hspace{-1cm}\mbox{with}\quad \Big \{ \begin{array}{lcl} \widetilde M _{\widetilde \theta_0}\big ( (x_{t-i})_{i\geq 1}\big )&=&M _{\widetilde \theta_0}\big ((\Lambda_{\beta_0}^{-1}(L) x_{t-i})_{i\geq 1}\big ) \\ \widetilde f _{\widetilde \theta_0}\big ( (x_{t-i})_{i\geq 1}\big )&=&f _{\theta_0}\big ((\Lambda_{\beta_0}^{-1}(L) x_{t-i})_{i\geq 1}\big ) -\sum_{j=1}^\infty \psi_j(\beta_0) x_{t-j}
\end{array}.
\end{eqnarray}
Finally, for $i=0$,
\begin{eqnarray}
   \nonumber\big |\widetilde f _{\widetilde \theta_0}\big ( (x_{t-i})_{i\geq 1}\big )-\widetilde f _{\widetilde\theta_0}\big ( (y_{t-i})_{i\geq 1}\big ) \big | & \leq & \sum_{j=1}^\infty  \alpha_j^{(0)}(f,\{\theta_0\})  \, \big | (\Lambda^{-1}_{\beta_0}(L) x_{t-j-i})_{i\geq 1}-(\Lambda_{\beta_0}^{-1}(L) y_{t-j-i})_{i\geq 1}  \big | \\
\nonumber 	&& \hspace{7cm}+|\psi_j(\beta_0)| \, | x_{t-j} -y_{t-j} | \\
  \nonumber & \leq & \sum_{j=1}^\infty  \alpha_j^{(0)}(f,\{\theta_0\})  \, \big |\sum_{k=0}^\infty |\psi_k(\beta_0) | \, |x_{t-k-j} -y_{t-k-j} \big |  +|\psi_j(\beta_0)| \, | x_{t-j} -y_{t-j} |  \\
 \nonumber & \leq & \sum _{j=1}^\infty \Big ( |\psi_j(\beta_0)| +\sum_{k=1}^j \alpha_k^{(0)}(f,\{\theta_0\}) \psi_{j-k} (\beta_0)\Big ) \, |x_{t-j} -y_{t-j} \big | \\
\label{alphaftilde} && \hspace{-2cm} \Longrightarrow \quad \alpha_j^{(0)}(\widetilde f,\{\widetilde \theta_0\})  \leq |\psi_j(\beta_0)| +\sum_{k=1}^j \alpha_k^{(0)}(f,\{\theta_0\}) \big | \psi_{j-k}(\beta_0) \big |.
\end{eqnarray}
Moreover, we also have:
\begin{eqnarray}
 \nonumber \big |\widetilde M _{\widetilde \theta_0}\big ( (x_{t-i})_{i\geq 1}\big )-\widetilde M _{\widetilde \theta_0}\big ( (y_{t-i})_{i\geq 1}\big ) \big | & \leq & \sum_{j=1}^\infty  \alpha_j^{(0)}(M,\{\theta_0\})  \, \big | (\Lambda_{\beta_0}^{-1}(L) x_{t-j-i})_{i\geq 1}-(\Lambda_{\beta_0}^{-1}(L) y_{t-j-i})_{i\geq 1}  \big |
  \\
\label{alphaMtilde} && \hspace{-2cm} \Longrightarrow \quad \alpha_j^{(0)}(\widetilde M,\{\widetilde \theta_0\})   \leq \sum_{k=1}^j \alpha_k^{(0)}(M,\{\theta_0\})  \big | \psi_{j-k} \big |.
\end{eqnarray}
The same kinds of computations could also be done by considering the first and second derivatives of $\widetilde f$ and $\widetilde M$ with respect to $\widetilde \theta$. Note, and this is important, that the first and second derivatives of $\Lambda^{-1}_{\beta}$ with respect to $\widetilde \theta$ are also EDS. Finally,
\begin{itemize}
\item if when $j \to \infty$, $\alpha_j^{(0)}(K,\{\theta_0\})=O(j^{-\beta})$ with $\beta>1$ and $\psi_{j}=O(\rho^j)$ with $0\leq \rho <1$, then there exists $C>0$ such as  $\sum_{k=1}^j \alpha_k^{(0)}(K,\{\theta_0\})  \big | \psi_{j-k} \big |\leq C \, \sum_{k=1}^j k^{-\beta} \rho ^{j-k}\sim -C (\log \rho)^{-1} j^{-\beta}$ and therefore $\alpha_j^{(0)}(\widetilde K,\{\widetilde \theta_0\}) =O(j^{-\beta})$.
\item if when $j \to \infty$,  $\alpha_j^{(0)}(K,\{\theta_0\})=O(r^{j})$ with $0 \leq r<1$ and $\psi_{j}=O(\rho^j)$ with $0\leq \rho <1$, then there exists $C>0$ such as  $\sum_{k=1}^j \alpha_k^{(0)}(K,\{\theta_0\})  \big | \psi_{j-k} \big |\leq C \, \sum_{k=1}^j r^{-k} \rho ^{j-k}=O(j \, \max(r,\rho)^j)$ and therefore $\alpha_j^{(O)}(\widetilde K,\{\widetilde \theta_0\})$ is EDS.
\end{itemize}
The same kind of computation can be also done for $(\alpha_j^{(i)}(\widetilde K,\{\widetilde \theta_0\}))_j$ since the derivatives and second-derivatives of $\Lambda^{-1}_{\beta_0}$ with respect to $\beta$ and therefore to $ \widetilde \theta$ are also EDS.
\end{proof}
Now we remind two lemmas already proved in \cite{BW}:
\begin{lemma}\label{lem::estvar}
Assume that $\theta_0\in \Theta(r)$ for $r \geq 1$ and $X$ is the causal stationary
solution of the equation \eqref{eq::sys}. If { $(A_0(K,\Theta))$} holds (with $K=f$ or $K=M$) then $K_\theta^t\in L^r({\cal C}(\Theta,R^m))$ and there exists $C>0$ not depending on $t$ such that
\begin{equation}\label{eq::estf}
\E \big [\|\widehat K_\theta^t-K_\theta^t\|_{\Theta}^r \big ] \le C \,
\E\big [|X_0|^r \big ]\Big (\sum_{j\ge t}\alpha_j(K,\Theta)\Big )^r\qquad \mbox{for all $t \in N^*$.}
\end{equation}
\end{lemma}

\begin{lemma}\label{lem::prel}
Let ${\cal D}^{(2)}(\Theta)$ denote the Banach space of $2$ times continuously differentiable functions on $\Theta$ equipped with the uniform norm
$$
\displaystyle \|h\|_{2,\Theta}=\|h\|_\Theta+ \Big
\|\frac {\partial h}{\partial \theta} \Big
\|_\Theta+\Big \|\frac {\partial^2
h}{\partial \theta \partial \theta'} \Big \|_\Theta.
$$
Let $\theta_0 \in \Theta(r)$ ($r\ge1$) and assume that for $i = 0,~1,~2,$ $(A_i (f,\Theta))$ and $(A_i (M,\Theta))$  hold. Then
$f_\theta^t\in \L^{r}\big ({\cal D}^{(2)}(\Theta)\big )$ and $M_\theta^t\in \L^{r}\big ({\cal D}^{(2)}(\Theta)\big ).$
\end{lemma}

\noindent Now, we begin with the proofs of Theorem \ref{SC}, \ref{Davis} and \ref{AN}.
\begin{proof}[Proof of Theorem \ref{SC}]

The proof of the theorem is divided into two parts and follows the same kind of procedure than in \cite{J}. In (i), a uniform (on $\Theta$) strong law of large numbers satisfied by $ \frac 1 n \, \widehat L_n(\theta)$  converging to $L(\theta) := -\E [q_0 (\theta)]$ is established. In (ii), it is proved that $L(\theta)$ admits a unique maximum in $\theta_0$. Those two conditions lead to the strong consistency of $\widehat{\theta}_n$ (from \cite{J}).\\
~\\
(i) In the same way and for the same reason in the proof of Theorem 1 of \cite{BW}, the uniform strong law of large numbers satisfied by the sample mean of $(\widehat{q}_t )_{t\in N^*}$ (defined in \eqref{eq::QMLE}]) is implied
by establishing $ \E [\|q_t (\theta)\|_\Theta] < \infty$. But new computations have to be done in case of Laplacian conditional log-density $q_t(\theta)$. From Lemma \ref{lem::estvar}, for all $t \in Z$,\\
\begin{align*}
	\begin{split}
|q_t(\theta)|	&= \left| {{ (M_\theta^t)^{-1} | X_t-f^t_{\theta}|}}+\log (M^t_\theta) \right|\\
    &\leq \frac{|X_t-f^t(\theta)|}{\underline{M}}+\big |\log(\underline{M})\big |+ M_\theta^t \\
		\Longrightarrow \qquad \sup_{\theta\in\Theta}|q_t(\theta)| & \leq \frac{1 }{\underline{M}} \big (|X_t| + \|f^t(\theta)\|_\Theta  \big ) +\big |\log(\underline{M})\big |+ \| M_\theta^t \|_\Theta.
	\end{split}					
\end{align*}
With $r\geq 1$, we have $\forall t \in \Z,~\E [|X_t|] < \infty$ from Proposition \ref{stationarity} and $\E \big  [\|f_\theta^t\|^r_\Theta+\|M_\theta^t\|^r_\Theta \big ] < \infty$ from Lemma \ref{lem::estvar}, implying $\E \big [\|f_\theta^t\|_\Theta+\|M_\theta^t\|_\Theta \big ] < \infty$. As a consequence, for all $t\in \Z$,
\begin{equation*}
\E \big [\|q_t(\theta) \|_\Theta \big ]< \infty.
\end{equation*}
Hence, the uniform strong law of large numbers for $(q_t (\theta))$ follows:
\begin{equation}\label{Las}
\Big\|\frac{L_n(\theta)}{n}-L(\theta)\Big\|_\Theta\limitepsn  0 .
\end{equation}
Now, we are going to establish  $\frac{1}{n} \big \|\widehat L_n(\theta)-L_n(\theta) \big \|_\Theta\limitepsn  0$. Indeed, for all $\theta \in \Theta$ and $t\in N^*$,
\begin{align*}
	\begin{split}
|\widehat{q_t}(\theta)-q_t(\theta)|	&\leq \big|\log({\widehat{M}}_\theta^t)-\log({M}_\theta^t)+ (\widehat{M}_\theta^t)^{-1} | X_t-\widehat{f}^t_{\theta}|-(M_\theta^t)^{-1} | X_t-f^t_{\theta}|\big|\\
                &\leq {\big|{\widehat{M}_\theta}^t-{{M}_\theta}^t\big |}{\underline{M}^{-1}}+ |\widehat{M}_\theta^t-M_\theta^t|  \underline{M}^{-2}|X_t-f^t_{\theta}|+\underline{M}^{-1} | \widehat{f}^t_{\theta}-f^t_{\theta}|\\
	\end{split}					
\end{align*}
with $C>0$.
Hence, we have:
\begin{align*}
	\begin{split}
\|\widehat{q_t}(\theta)-q_t(\theta)\|_\Theta	\leq C\big( 1+ |X_t |+\|{f}_\theta^t\|_\Theta \big)\big(\|\widehat{M}_\theta^t-{M}_\theta^t\|_\Theta+\|\widehat{f}_\theta^t-{f}_\theta^t\|_\Theta\big).
	\end{split}					
\end{align*}
By Corollary 1 of \cite{KW}, the proof is achieved if there exists $s \in (0,1]$ such as
\begin{equation}\label{eq::pp}
\sum_{t\geq1} {{\frac{1}{t^s}}}\,  \E \big [\|q_t(\theta)-\widehat{q}_t(\theta)\|_\Theta^{s}\big ]<  \infty.
\end{equation}
Let us prove \eqref{eq::pp} with $s=r/2$ when $r \in [1,2]$.\\
From Cauchy-Schwarz Inequality and assumptions $A_0(f, {\Theta})$ and $A_0(M, {\Theta})$,
\begin{align*}
	\begin{split}
\E \big [\|\widehat{q_t}(\theta)-q_t(\theta)\|_\Theta^{r/2} \big]	\leq C \, \big ( \E \big[ (1+ |X_t |+\|{f}_\theta^t\|_\Theta)^{r}\big ] \big )^\frac{1}{2}  \, \big ( \E \big [(\|\widehat{M}_\theta^t-{M}_\theta^t\|_\Theta+\|\widehat{f}_\theta^t-{f}_\theta^t\|_\Theta)^{r}\big] \big )^\frac{1}{2}.
	\end{split}					
\end{align*}
Using Lemma \ref{lem::estvar} and previous proved results implying $\E [|X_t|^r]< \infty$, $\E [\|{f}_\theta^t\|_\Theta^r+\|{M}_\theta^t\|_\Theta^r]< \infty$, we obtain
\begin{eqnarray*}
\E \big [\|\widehat{q_t}(\theta)-q_t(\theta)\|_\Theta^{r/2} \big ]	&\leq & C \,  \Big (\sum_{j>t} \alpha_j^{(0)}(f,\Theta)+\alpha_j^{(0)}(M,\Theta)\Big)^\frac{r}{2} \\
& \leq & C \,  t^{-\frac{(\ell-1)r}{2}},		
\end{eqnarray*}
where the last inequality is obtained from the condition \eqref{decSC} of Theorem \ref{SC}. \\
Hence, we have
\begin{align*}
	\begin{split}
\sum_{t\geq 1} {{\frac{1}{t^{r/2}}}} \, \E \big  [|\widehat{q_t}(\theta)-q_t(\theta)|_\Theta^{r/2} \big]	&    \leq A \sum_{t\geq 1}t^{-r \, \ell/2},
	\end{split}					
\end{align*}
which is finite when $r \, \ell >2$. When $r \geq 2$, it is sufficient to consider the case $r=2$. As a consequence, we obtain
\begin{equation}\label{majq}
\frac 1 n \, \sum_{t=1}^n  \big \|\widehat{q_t}(\theta)-q_t(\theta) \big \|_\Theta \limitepsn 0\quad \mbox{and} \quad \frac{1}{n}  \, \big \|\widehat L_n(\theta)-L_n(\theta) \big \|_\Theta\limitepsn  0,
\end{equation}
and therefore, using \eqref{Las},
\begin{equation}\label{convLhat}
\frac{1}{n}  \, \big \|\widehat L_n(\theta)-L(\theta) \big \|_\Theta\limitepsn  0.
\end{equation}
(ii) Now for $\theta \in \Theta$, we study
$$
L(\theta)=-\E  [q_0(\theta)].
$$
which can also be consider as a Kullback-Leibler discripency.
We have
\begin{eqnarray*}
L(\theta)&=& - \E\big  [\log (M_{\theta}^t)+(M_{\theta}^t)^{-1} \, \big |X_t-f_{\theta}^t \big | \big ] \\
&=&- \E\Big  [\log (M_{\theta}^t)+\frac {M_{\theta_0}^t} {M_{\theta}^t} \, \Big |\zeta_t+ \frac {f_{\theta_0}^t-f_{\theta}^t }{M_{\theta_0}^t} \Big | \Big ].
\end{eqnarray*}
Hence, using $\E  [ |\zeta_t|]=1$, we obtain:
\begin{eqnarray*}
L(\theta_0)-L(\theta)
&=&\E\Big  [ \log \big (\frac {M_{\theta}^t}{M_{\theta_0}^t} \big )+\frac {M_{\theta_0}^t} {M_{\theta}^t} \, \Big |\zeta_t+ \frac {f_{\theta_0}^t-f_{\theta}^t }{M_{\theta_0}^t} \Big | -1\Big ] \\
&=&\E \Big [\log \big (\frac {M_{\theta}^t}{M_{\theta_0}^t} \big )-1 + \frac {M_{\theta_0}^t} {M_{\theta}^t} \E\Big [\Big |\zeta_t+ \frac {f_{\theta_0}^t-f_{\theta}^t }{M_{\theta_0}^t} \Big | ~|~(X_{t-k})_{k\geq 1} \Big ] \Big ].
\end{eqnarray*}
But for $\zeta_t$ following a symmetric probability distribution, for any $m\in \R^*$, $\E [|\zeta_t+m|]> \E [|\zeta_t|]=1$. Therefore, for $\theta\neq \theta_0$, if $f_\theta\neq f_{\theta_0}$ (else $>$ is replaced by $\geq$),
\begin{eqnarray*}
L(\theta_0)-L(\theta)
&> &\E \Big [\log \big (\frac {M_{\theta}^t}{M_{\theta_0}^t} \big )-1 + \frac {M_{\theta_0}^t} {M_{\theta}^t} \Big ] \\
& > & h\Big (\frac {M_{\theta_0}^t}{M_{\theta}^t} \Big ),
\end{eqnarray*}
with $h(x)=-\log(x)-1+x$. But for any $x\in (0,1)\cup(1,\infty)$, $h(x)> 0$ and $h(1)=0$. Therefore if $M_\theta\neq M_{\theta_0}$, $h\Big (\frac {M_{\theta_0}^t}{M_{\theta}^t} \Big )>0$ ($>0$ is replaced by $=0$ if $M_\theta= M_{\theta_0}$). This implies from Condition C3 (Identifiability) that $ L(\theta_0)-L(\theta)> 0$ almost surely for all $\theta \in \Theta$, $\theta\neq \theta_0$. Hence a supremum of $L(\theta)$ is only reached for $\theta=\theta_0$ which is the unique maximum.
\end{proof}
\begin{proof}[Proof of Theorem \ref{Davis}]
We follow the same scheme of proof than in \cite{DD}. Hence, denote
\begin{eqnarray*}
S_n&=&\sum_{t=1}^n V_{t-1} \big (| Z_t-n^{-1/2} Y_{t-1} |-|Z_t| \big )  \\
&= & -n^{-1/2} \sum_{t=1}^n V_{t-1}Y_{t-1}  \mbox {sgn} (Z_t) \\
&& \hspace{2cm}+ 2  \sum_{t=1}^n V_{t-1}  ( n^{-1/2} Y_{t-1}-Z_t)\big ( \1 _{0 <Z_t < n^{-1/2} Y_{t-1}} -\1 _{n^{-1/2} Y_{t-1} <Z_t < 0}\big ) \\
& =& A_n + B_n.
\end{eqnarray*}
Since $\E \big [ V_{t-1}Y_{t-1}  \mbox {sgn} (Z_t) ~|~{\cal F}_{t-1} \big ]=\E [ \mbox {sgn} (Z_t)  ] \,  \E \big [ V_{t-1}Y_{t-1} \big ]=0$ and  $\E \big [ V_{0}^2Y_{0}^2 \big ]< \infty$, we can apply a central limit theorem for stationary martingale difference sequence (see \cite{B}) and
\begin{equation}\label{An}
A_n  \limiteloin {\cal N} \big ( 0 \, , \, \E \big [ V_{0}^2Y_{0}^2 \big ] \big ).
\end{equation}
Now, considering $B_n$, define also $W_{nt} =V_{t-1}  ( n^{-1/2} Y_{t-1}-Z_t)\, \1 _{0 <Z_t < n^{-1/2} Y_{t-1}} $. Using the same arguments as in \cite{DD}, we also obtain
\begin{eqnarray*}
&& \bullet \quad \limsup_{n\to \infty} n \, \E \big [ W_{nt} ^2 \big ]=0; \\
&& \bullet \quad \E \big [ W_{nt} ~|~{\cal F}_{t-1} \big ] \simeq \frac 1 {2n} \, f(0) \, V_{t-1} Y_{t-1}^2\quad \mbox{for $|n^{-1/2} Y_{t-1} | <\varepsilon$} ; \\
&& \bullet \quad \sum_{t=1}^n W_{nt}  \limiteproba \frac 1 2 \, f(0) \, \E \big [ V_0 Y_0^2 \, \1 _{Y_0>0} \big ].
\end{eqnarray*}
Then we deduce
\begin{equation}\label{Bn}
B_n \limiteproba f(0) \,  \E \big [ V_0 Y_0^2  \big ].
\end{equation}
The proof is achieved from \eqref{An} and \eqref{Bn}.
\end{proof}

\begin{proof}[Proof of Theorem \ref{AN}]
We follow a proof which is similar to the one of Theorem 2 in \cite{DD} or \cite{LL}. \\
Let $v =\sqrt n (\theta-\theta_0) \in \R^d$. Then we are going to prove in 2/ that maximizing $\widehat L_n(\theta)$ is equivalent to maximizing $L_n(\theta)$ which is equivalent to maximizing
\begin{eqnarray}  \label{eq::m}
W_n(v)&=&-\sum_{t=1}^n \big (q_t(\theta_0+n^{-1/2} v)-q_t(\theta_0)\big )\\
\nonumber
 &=& \sum_{t=1}^n \log\Big (\frac { ({M}_{\theta_0+n^{-1/2} v}^t)^{-1}}  { ({M}_{\theta_0}^t)^{-1}}  \Big )+({ {M}}_{\theta_0}^t)^{-1} | X_t-f^t_{\theta_0}|- ({ {M}}_{\theta_0+n^{-1/2} v}^t)^{-1} \big | X_t-{ {f}}^t_{\theta_0+n^{-1/2} v} \big |
\end{eqnarray}
with respect to $v$. As a consequence, there exists a sequence  $(\widehat v_n)_n$ where $\widehat v_n$ is a maximizer of $W_n(v)$ such as $\widehat v_n=\sqrt n (\widehat \theta_n-\theta_0)$. In 1/ we will provide a limit theorem satisfied by $W_n(v)$.   Then we are going to prove in 3/ that $(W_n(\cdot ))_n$ converges as a process of ${\cal C}(\R^d)$ (space of continuous functions on $\R^d$) to a limit process $W$. Hence $(\widehat v_n)_n$ converges to the maximizer of $W$. \\
~\\
1/ First, we are going to study the asymptotic behavior of $W_n(v)$. We have
\begin{eqnarray*}
W_n(v)&=& \sum_{t=1}^n \log\Big (\frac { ({M}_{\theta_0+n^{-1/2} v}^t)^{-1}}  { ({M}_{\theta_0}^t)^{-1}}  \Big ) +| X_t-f^t_{\theta_0}| \big (({M}_{\theta_0}^t)^{-1}-({ {M}}_{\theta_0+n^{-1/2} v}^t)^{-1} \big )\\&& +\sum_{t=1}^n ({ {M}}_{\theta_0+n^{-1/2} v}^t)^{-1}  \big ( | X_t-f^t_{\theta_0}|- \big | X_t-{ {f}}^t_{\theta_0+n^{-1/2} v} \big | \big )  \\
&=& I_1(v) +I_2(v).
\end{eqnarray*}
We have:
\begin{eqnarray*}
I_1(v)&=& -\sum_{t=1}^n \log\Big (\frac { {M}_{\theta_0+n^{-1/2} v}^t}  { {M}_{\theta_0}^t}  \Big ) +| \zeta_t| \Big (1-\frac{ {M}_{\theta_0}^t}  { {M}_{\theta_0+n^{-1/2} v}^t}   \Big )
\end{eqnarray*}
Using Taylor expansions, we deduce that for each $t \in \{1,\cdots,n\}$, there exists $\overline \theta_1^t$ and $\overline \theta_2^t$ in the segment $[\theta_0,\theta_0+n^{-1/2} v]$ such as:
\begin{eqnarray*}
\log\Big (\frac { {M}_{\theta_0+n^{-1/2} v}^t}  { {M}_{\theta_0}^t}  \Big )& =& n^{-1/2}({M}_{\theta_0}^t)^{-1} v' \,  \Big (\frac {\partial M^t_\theta}{\partial \theta} \Big )_{\theta_0} + \frac 12 \, n^{-1} \Big \{   v' \,  \Big (\frac {\partial^2 M^t_\theta}{\partial \theta^2} \Big )_{\overline \theta_1^t}- ({M}_{\theta_0}^t)^{-2}\Big (v' \,  \Big (\frac {\partial M^t_\theta}{\partial \theta} \Big )_{\overline \theta_1^t}\Big )^2 \Big  \} \\
\frac{ {M}_{\theta_0}^t}  { {M}_{\theta_0+n^{-1/2} v}^t} & =& 1- n^{-1/2}({M}_{\theta_0}^t)^{-1} v' \,  \Big (\frac {\partial M^t_\theta}{\partial \theta} \Big )_{\theta_0}+ \frac 12 \, n^{-1} \Big \{  2({M}_{\theta_0}^t)^{-2}\Big (v' \,  \Big (\frac {\partial M^t_\theta}{\partial \theta} \Big )_{\overline \theta_2^t}\Big )^2 - v' \,  \Big (\frac {\partial^2 M^t_\theta}{\partial \theta^2} \Big )_{\overline \theta_2^t}\Big  \}
\end{eqnarray*}
Then,
\begin{eqnarray*}
I_1(v)&=& n^{-1/2} \sum_{t=1}^n ({M}_{\theta_0}^t)^{-1} v' \,  \Big (\frac {\partial M^t_\theta}{\partial \theta} \Big )_{\theta_0} \big (  | \zeta_t|-1 \big ) + \frac 1 {2n}\sum_{t=1}^n     ({M}_{\theta_0}^t)^{-1}\Big \{  v' \,  \Big (\frac {\partial^2 M^t_\theta}{\partial \theta^2} \Big )_{\overline \theta_2^t} v \, | \zeta_t| - v' \,  \Big (\frac {\partial^2 M^t_\theta}{\partial \theta^2} \Big )_{\overline \theta_1^t} v\Big \} \\
&& \hspace{3cm} + \frac 1 {2n}\sum_{t=1}^n     ({M}_{\theta_0}^t)^{-2}\Big  \{ \Big (v' \,  \Big (\frac {\partial M^t_\theta}{\partial \theta} \Big )_{\overline \theta_1^t}\Big )^2 -2 \,  \Big (v' \,  \Big (\frac {\partial M^t_\theta}{\partial \theta} \Big )_{\overline \theta_2^t}\Big )^2| \zeta_t|  \Big \} \\
&=& I_1^{(1)}(v)+ I_1^{(2)}(v)+ I_1^{(3)}(v).
\end{eqnarray*}
Using a Central Limit Theorem for martingale-differences (see for instance \cite{B}), and since from Lemma \ref{lem::prel}, $\E \big [\big \| ({M}_{\theta_0}^t)^{-1} v' \,  \Big (\frac {\partial M^t_\theta}{\partial \theta} \Big )_{\theta_0} \big \|^2_\Theta\big ]<\infty$ and $\E \big [({M}_{\theta_0}^t)^{-1} v' \,  \Big (\frac {\partial M^t_\theta}{\partial \theta} \Big )_{\theta_0} \big (  | \zeta_t|-1 \big )~|~{\cal F}_{t-1} \big ]=0$, we have:
\begin{equation}\label{I11}
I_1^{(1)}(v) \limiteloin {\cal N} \Big (0 \, , \, \E \Big [ ({M}_{\theta_0}^0)^{-2} \Big ( v' \,  \Big (\frac {\partial M^0_\theta}{\partial \theta} \Big )_{\theta_0}\Big )^2 \Big ] \big ( \sigma^2_\zeta -1 \big ) \Big ).
\end{equation}
Now, using that $\theta  \in \Theta \mapsto \frac {\partial M^t_\theta}{\partial \theta}$ and $\theta \mapsto \frac {\partial^2 M^t_\theta}{\partial \theta^2}$ are continuous functions, $\overline \theta_1^t \limiteloin  \theta_0$ and $\overline \theta_2^t \limiteloin \theta_0$, we claim that $I_1^{(2)}(v)$ have the same limit distribution that $\frac 1 {2n}\sum_{t=1}^n     ({M}_{\theta_0}^t)^{-1} v' \,  \Big (\frac {\partial^2 M^t_\theta}{\partial \theta^2} \Big )_{\theta_0}v \, \big (| \zeta_t| -1 \big )$. From Lemma \ref{lem::prel}, note that $\E \Big  [({M}_{\theta_0}^t)^{-1} v' \,  \Big (\frac {\partial^2 M^t_\theta}{\partial \theta^2} \Big )_{\theta_0}v \, \big (| \zeta_t| -1 \big ) ~|~{\cal F}_{t-1} \Big ]=0$ and $\E \big [\big \| ({M}_{\theta_0}^t)^{-1} v' \,  \Big (\frac {\partial^2 M^t_\theta}{\partial \theta^2} \Big )_{\theta_0} v \big \| \big ]<\infty$. \\ Thus, from the strong large number law  for martingale-differences (see again \cite{B}), we obtain:
$$
\frac 1 {2n}\sum_{t=1}^n     ({M}_{\theta_0}^t)^{-1} v' \,  \Big (\frac {\partial^2 M^t_\theta}{\partial \theta^2} \Big )_{\theta_0}v \, \big (| \zeta_t| -1 \big ) \limitepsn 0,
$$
and this implies:
\begin{equation}\label{I12}
I_1^{(2)}(v) \limiteloin 0.
\end{equation}
Previous arguments induce that $I_1^{(3)}(v)$ has the same limit distribution that $\displaystyle \frac 1 {2n}\sum_{t=1}^n  ({M}_{\theta_0}^t)^{-2}   \Big  (v' \,  \Big (\frac {\partial M^t_\theta}{\partial \theta} \Big )_{\theta_0}\Big )^2 \big (1-2| \zeta_t| \big )$. From the strong large number law  for martingale-differences (see \cite{B}), we obtain:
\begin{eqnarray*}
\frac 1 {2n}\sum_{t=1}^n       ({M}_{\theta_0}^t)^{-2} \Big (v' \,  \Big (\frac {\partial M^t_\theta}{\partial \theta} \Big )_{\theta_0}\Big )^2\, \big (1-2| \zeta_t| \big ) &\limitepsn & \frac 1 2 \,  \E \Big [ ({M}_{\theta_0}^0)^{-2}  \Big (v' \,  \Big (\frac {\partial M^0_\theta}{\partial \theta} \Big )_{\theta_0}\Big )^2\, \big (1-2| \zeta_0| \big ) \Big ] \\
& \limitepsn & -\frac 1 2 \,  \E \Big [({M}_{\theta_0}^0)^{-2}   \Big (v' \,  \Big (\frac {\partial M^0_\theta}{\partial \theta} \Big )_{\theta_0}\Big )^2\Big ],
\end{eqnarray*}
and this implies:
\begin{equation}\label{I13}
I_1^{(3)}(v) \limiteloin -\frac 1 2 \,  \E \Big [ ({M}_{\theta_0}^0)^{-2}   \Big (v' \,  \Big (\frac {\partial M^0_\theta}{\partial \theta} \Big )_{\theta_0}\Big )^2\Big ].
\end{equation}
Finally, from \eqref{I11}, \eqref{I12} and \eqref{I13}, we obtain:
\begin{equation}\label{I1}
I_1(v) \limiteloin {\cal N} \Big (-\frac 1 2 \,  \E \Big [ ({M}_{\theta_0}^0)^{-2}  \Big (v' \,  \Big (\frac {\partial M^0_\theta}{\partial \theta} \Big )_{\theta_0}\Big )^2\Big ]\, , \, \E \Big [ ({M}_{\theta_0}^0)^{-2} \Big ( v' \,  \Big (\frac {\partial M^0_\theta}{\partial \theta} \Big )_{\theta_0}\Big )^2 \Big ] \big ( \sigma^2_\zeta -1 \big ) \Big ).
\end{equation}
Now we consider $\displaystyle I_2(v)= \sum_{t=1}^n ({ {M}}_{\theta_0+n^{-1/2} v}^t)^{-1}  \big ( | X_t-f^t_{\theta_0}|- \big | X_t-{ {f}}^t_{\theta_0+n^{-1/2} v} \big | \big )$. Using again Taylor expansion, we can write:
\begin{eqnarray*}
I_2(v)&=&\sum_{t=1}^n ({ {M}}_{\theta_0}^t)^{-1}  \big ( | X_t-f^t_{\theta_0}|- \big | X_t-{ {f}}^t_{\theta_0+n^{-1/2} v} \big | \big ) \\
&& \hspace{2cm}-n^{-1/2} \sum_{t=1}^n ({ {M}}_{\overline \theta^t_M}^t)^{-2}  \, v' \,  \Big (\frac {\partial M^t_\theta}{\partial \theta} \Big )_{\overline \theta^t_M} \big ( | X_t-f^t_{\theta_0}|- \big | X_t-{ {f}}^t_{\theta_0+n^{-1/2} v} \big | \big ) \\
&=& I_2^{(1)}(v)+I_2^{(2)}(v),
\end{eqnarray*}
with $\overline \theta^t_M$ in the segment $[\theta_0, \theta_0+n^{-1/2} v]$. \\
First we have:
\begin{eqnarray*}
I_2^{(1)}(v)=\sum_{t=1}^n  \big ( | \zeta_t|- \big | \zeta_t -n^{-1/2}  ({ {M}}_{\theta_0}^t)^{-1}v' \,  \big ( \frac {\partial  f^t_{\theta} }{\partial \theta} \big )_{\overline \theta_f^t} \big | \big )
\end{eqnarray*}
with $\overline \theta^t_f$ in the segment $[\theta_0, \theta_0+n^{-1/2} v]$. Using Theorem \ref{Davis}, which an extension of Theorem 1 established in \cite{DD}, denoting $Z_t=\zeta_t$, $Y_t= ({ {M}}_{\theta_0}^t)^{-1}v' \,  \big ( \frac {\partial  f^t_{\theta} }{\partial \theta} \big )_{\overline \theta_f^t}$ and $V_t=1$ for $t \in \Z$,
\begin{equation}\label{I21}
I_2^{(1)} \limiteloin {\cal N} \Big ( -g(0) \,  \E \big [ ({ {M}}_{\theta_0}^0)^{-2} \big (v' \,  \big ( \frac {\partial  f^0_{\theta} }{\partial \theta} \big )_{\theta_0} \big) ^2   \big ] \, , \, \E \big [ ({ {M}}_{\theta_0}^0)^{-2} \big (v' \,  \big ( \frac {\partial  f^0_{\theta} }{\partial \theta} \big )_{\theta_0} \big) ^2   \big ]   \Big ).
\end{equation}
since $\E \big [ Y_t^2V_t^2 \big ] \leq \underline M^{-2 } \E \big [ \big \| v' \,   \frac {\partial  f^t_{\theta} }{\partial \theta} \big \|^2 _\Theta \big ]<\infty $ from Lemma \ref{lem::prel}. \\
Then, we have:
\begin{eqnarray}
\nonumber I_2^{(2)}(v) & \sim &n^{-1/2}  \sum_{t=1}^n ({ {M}}_{\theta_0}^t)^{-2}  \, v' \,  \Big (\frac {\partial M^t_\theta}{\partial \theta} \Big )_{\theta_0} \big ( | X_t-f^t_{\theta_0}|- \big | X_t-{ {f}}^t_{\theta_0+n^{-1/2} v} \big | \big ) \\
\nonumber &\sim & n^{-1/2}  \sum_{t=1}^n (M_{\theta_0}^t)^{-1}  \, v' \,  \Big (\frac {\partial M^t_\theta}{\partial \theta} \Big )_{\theta_0} \big ( | \zeta_t|- \big |\zeta_t-n^{-1/2}(M_{\theta_0}^t)^{-1}   v'\big (\frac {\partial f^t_\theta}{\partial \theta} \big )_{\overline \theta^t_f}  \big | \big ) \\
\label{I22}& \limiteproba & 0,
\end{eqnarray}
using the proof of Theorem \ref{Davis} and denoting $Z_t=\zeta_t$, $Y_t= ({ {M}}_{\theta_0}^t)^{-1}v' \,  \big ( \frac {\partial  f^t_{\theta} }{\partial \theta} \big )_{\overline \theta_f^t}$ and $V_t=(M_{\theta_0}^t)^{-1}  \, v' \,  \Big (\frac {\partial M^t_\theta}{\partial \theta} \Big )_{\theta_0}$ for $t \in \Z$ and condition $\E \big [|V_t Y_t| \big ]<\infty$ insuring a strong law of large number instead of central limit theorem for a martingale difference process. Therefore, from \eqref{I21} and \eqref{I22}, we deduce
\begin{equation}\label{I2}
I_2(v) \limiteloin {\cal N} \Big (-g(0) \,  \E \big [ ({ {M}}_{\theta_0}^0)^{-2} \big (v' \,  \big ( \frac {\partial  f^0_{\theta} }{\partial \theta} \big )_{\theta_0} \big) ^2   \big ]  \, , \,  \E \big [ ({ {M}}_{\theta_0}^0)^{-2} \big (v' \,  \big ( \frac {\partial  f^0_{\theta} }{\partial \theta} \big )_{\theta_0} \big) ^2   \big ]   \Big ).
\end{equation}
Finally, we obtain the behavior of $W_n(v)$ defined in \eqref{eq::m} from  \eqref{I1} and \eqref{I2}. However, we have to specify the asymptotic dependency relation between $I_1^{(1)}$ and $I_2^{(1)}$. Indeed these two terms converge to a Gaussian law. This implies to consider the asymptotic behavior of the sum of these two terms which could be reduced to the asymptotic behavior of:
\begin{eqnarray*}
n^{-1/2} \sum_{t=1}^n ({M}_{\theta_0}^t)^{-1} v' \,  \Big (\frac {\partial M^t_\theta}{\partial \theta} \Big )_{\theta_0} \big (  | \zeta_t|-1 \big ) + n^{-1/2} \sum_{t=1}^n  ({ {M}}_{\theta_0}^t)^{-1}v' \,  \big ( \frac {\partial  f^t_{\theta} }{\partial \theta} \big )_{\theta_0} \,   \mbox {sgn} (\zeta_t),
\end{eqnarray*}
from the proof of Theorem \ref{Davis}. Using again a central limit theorem for martingale differences, we obtain as asymptotic variance:
\begin{eqnarray}
\nonumber
\E \Big [
(M_{\theta_0}^t)^{-2} \Big ( v' \,  \Big (\frac {\partial M^t_\theta}{\partial \theta} \Big )_{\theta_0} \big (  | \zeta_t|-1 \big )
  + v' \,  \Big ( \frac {\partial  f^t_{\theta} }{\partial \theta} \Big )_{\theta_0} \,   \mbox {sgn} (\zeta_t)\Big )^2 \Big ]& =& \E \Big [ ({M}_{\theta_0}^0)^{-2} \Big ( v' \,  \Big (\frac {\partial M^0_\theta}{\partial \theta} \Big )_{\theta_0}\Big )^2 \Big ] \big ( \sigma^2_\zeta -1 \big ) \\
\nonumber &&\hspace{-9cm}+2 \, \E \Big [(M_{\theta_0}^0)^{-2} v' \,  \Big (\frac {\partial M^t_\theta}{\partial \theta} \Big )_{\theta_0} v' \,  \Big ( \frac {\partial  f^0_{\theta} }{\partial \theta} \Big )_{\theta_0}  \Big ]\, \E \big [ \big (  | \zeta_t|-1 \big ) \mbox {sgn} (\zeta_t) \big ]+ \E \Big [(M_{\theta_0}^0)^{-2} \Big (v' \,  \Big ( \frac {\partial  f^0_{\theta} }{\partial \theta} \Big )_{\theta_0} \Big) ^2   \Big ]\\
&&\hspace{-9cm}=\E \Big [ ({M}_{\theta_0}^0)^{-2} \Big \{ \big ( \sigma^2_\zeta -1 \big ) \Big ( v' \,  \Big (\frac {\partial M^0_\theta}{\partial \theta} \Big )_{\theta_0}\Big )^2+\Big (v' \,  \Big ( \frac {\partial  f^0_{\theta} }{\partial \theta} \Big )_{\theta_0} \Big) ^2  \Big \} \Big ]
\label{asymptovar}\end{eqnarray}
since $(\zeta_t)_t$ admits a symmetric probability distribution with a null median and expectation.  Therefore, there is no covariance term and finally we obtain:
\begin{eqnarray}\label{asymptoW}
W_n(v) \limiteloin W(v)=v' \, \big ( -\frac 1 2 \, \Gamma_M -g(0)\, \Gamma_F \big ) \, v +v' \, N \qquad  \\
\label{Gamma}
\mbox{with}
\qquad \left \{ \begin{array}{l}
N \egalloi { \cal N} \big (0 , \, \big ( \big ( \sigma^2_\zeta -1 \big ) \, \Gamma_M + \Gamma_F \big )  \big ) \\
\Gamma_F= \Big (  \E \Big [({M}_{\theta_0}^0)^{-2}\Big ( \frac {\partial f_\theta^0}{\partial \theta_i}\Big ) _{\theta_0} \Big ( \frac {\partial f_\theta^0}{\partial \theta_j}\Big ) _{\theta_0}   \Big ] \Big )_{1\leq i,j\leq d} \\
\Gamma_M= \Big (  \E \Big [\Big ( \frac {\partial \log (M_\theta^0)}{\partial \theta_i}\Big ) _{\theta_0} \Big ( \frac {\partial \log (M_\theta^0)}{\partial \theta_j}\Big ) _{\theta_0}   \Big ] \Big )_{1\leq i,j\leq d}
\end{array}
\right .
\end{eqnarray}
~\\
2/ Now, we consider the approximation $\widehat W_n(v)$ of $W_n(v)$ defined by:
$$
\widehat W_n(v)=-\sum_{t=1}^n \big (\widehat q_t(\theta_0+n^{-1/2} v)-\widehat q_t(\theta_0)\big ) \qquad \mbox{for any $v \in \R^d$.}
$$
From the assumptions of Theorem \ref{SC} and \eqref{majq} we have $\frac 1 n \, \sum_{t=1}^n  \big \|\widehat{q_t}(\theta)-q_t(\theta) \big \|_\Theta \limitepsn 0$. Then we have $\widehat W_n(v)=W_n(v)+ R_n(v)$ with $ \big [ \sup_{v \in \R^d}  |R_n(v)| \big ] \leq 2 \, \sum_{t=1}^n \big [ \big \| \widehat q_t(\theta)-q_t(\theta) \big \|_\Theta \big ] \limitepsn 0$ and then:
\begin{eqnarray}\label{asymptoWhat}
\widehat W_n(v) \limiteloin W(v)
\end{eqnarray}
with $W$ defined in \eqref{asymptoW}. \\
~\\
3/ Now, from \eqref{asymptoWhat}, the proof of Theorem \ref{Davis} and the same arguments than in the proof of Theorem 2 of \cite{DD}, we deduce that finite distributions $(\widehat W_n(v_1),\cdots,\widehat W_n(v_k))$ converge to $(W(v_1),\cdots,W(v_k))$ for any $(v_1,\cdots,v_k)\in (\R^d)^k$. Moreover, always following the proof of Theorem 2, $(W_n(v))_v$ converges to $(W(v))_v$ as a process on the continuous function space ${\cal C}^0$. \\
As a consequence, a maximum $\widehat v$ of $\widehat W_n(v)$ satisfies:
$$
\widehat v=\big (  \Gamma_M +2g(0)\, \Gamma_F \big )^{-1} N,
$$
with $N$ defined in \eqref{Gamma} and this implies \eqref{tlcqmle}.
\end{proof}

\begin{proof}[Proof of Proposition \ref{QMLAPARCH}]
First, Condition {\bf C2} is satisfied since $b_0>0$. Other conditions on Lipschitz coefficients are also satisfied from Lemma \ref{cond2} (see the arguments above). The identifiability condition {\bf C3} is also satisfied from the following which are divided into two parts. In (i) we proof that $\big(\delta, b_0,(b_i^+(\theta),b_i^-(\theta))_{i\geq 1}\big)$ (defined in \eqref{bb}) are unique, thereafter in (ii) we proof that $\theta=\big(\omega,(\alpha_i)_{1\leq i\leq p},(\gamma_i)_{1\leq i\leq p},(\beta_i)_{1\leq i \leq q}\big)$ is also unique.\\
(i) The proof of this result follow the same reasoning in \cite{BHK}. First we have
\begin{eqnarray}\label{di}
  \sigma_t^\delta &=  b_0(\theta)+\sum_{i\geq1} b_i^+(\theta) (\max(X_{t-i},0))^\delta +\sum_{i\geq1} b_i^-(\theta) (\max(-X_{t-i},0))^\delta.
\end{eqnarray}
We prove the result by contradiction. Suppose that there exist two vectors $\beta=\big(\delta, b_0,(b_i^+)_{i\geq 1},(b_i^-)_{i\geq 1}\big)$ and $\beta'= \big(\delta',b_0',(b_i^{+'})_{i\geq 1},(b_i^{-'})_{i\geq 1}\big)$ verifying \eqref{di}. Let $m>0$ be the smallest integer satisfying $b_m^+\neq b_m^{+'}$ or $b_m^- \neq b_m^{-'}$ (if $b_i^+ = b_i^{+'}$ and $b_i^- = b_i^{-'} \ \ \forall i\geq 1$ then $b_0=b_0'$). In one hand, since $x \in (0,\infty)\mapsto x^\delta$ is a one-to-one map and since $\P(X_t=\pm 1,~\forall t \in \Z)=0$, we have $\delta =\delta'$. In the other hand, by definition of $m$, we have
\begin{eqnarray}\label{hhh}
\nonumber \hspace{-1.5cm}(b_m^{+'}-b_m^{+}) (\max(X_{t-i},0))^{\delta} + (b_m^{-'}-b_m^{-}) (\max(-X_{t-m},0))^\delta&& \\
&& \hspace{-8.5cm}
=b_0-b_0'+\sum_{i\geq m+1} (b_i^+-b_i^{+'}) (\max(X_{t-i},0))^\delta + \sum_{i\geq m+1} (b_i^--b_i^{-'}) (\max(-X_{t-i},0))^\delta.
\end{eqnarray}
From \eqref{aparch}, we have $X_{t-m}=\sigma_{t-m} \zeta_{t-m}$, therefore
\begin{eqnarray*}
(b_m^{+'}-b_m^{+}) (\max(X_{t-m},0))^\delta + (b_m^{-'}-b_m^{-}) (\max(-X_{t-m},0))^\delta= \left\{
  \begin{array}{ll}
     (b_m^{+'}-b_m^{+})    \sigma_{t-m}^\delta \zeta_{t-m}^\delta  & \hbox{when $\zeta_{t-m}\geq 0$ }\\
    (b_m^{-'}-b_m^{-})  \sigma_{t-m}^\delta (-\zeta_{t-m})^\delta  & \hbox{when $\zeta_{t-m}< 0$ }  \end{array}
\right.
\end{eqnarray*}
Moreover \eqref{hhh} and the fact that $b_m^+\neq b_m^{+'}$ or $b_m^- \neq b_m^{-'}$ implies that at least one of the following equalities hold
$$
\left\{
  \begin{array}{ll}
     \zeta_{t-m}^\delta =( (b_m^{+'}-b_m^{+})    \sigma_{t-m}^\delta)^{-1} \Big( \sum_{i\geq m+1} (b_i^+-b_i^{+'}) (\max(X_{t-i},0))^\delta\Big)\, & \hbox{when $\zeta_{t-m}\geq 0$ }\\
\qquad \qquad \mbox{or}\\
     (-\zeta_{t-m})^\delta = ((b_m^{-'}-b_m^{-})  \sigma_{t-m}^\delta)^{-1} \Big(\sum_{i\geq m+1} (b_i^--b_i^{-'}) (\max(X_{t-i},0))^\delta\Big)\, & \hbox{when $\zeta_{t-m}< 0$ }  \end{array}
\right.
$$
Since $\sigma_{t-m}^\delta>b_0>0, \ \zeta^\delta_{t-m}$  is well defined. Let $F_k$ be the $F$-algebra generated by $(\zeta_i, i<k)$. The causal representation of tha APARCH$(\delta,p,q)$ shows that $X_j$ is $F_j$-measurable and thus the right-hand side of the above equations  (and consequently also $\zeta_{t-m}^\delta$ in the case $\zeta_{t-m}\geq 0$ or the case $\zeta_{t-m}< 0$) is a real-valued random variable, measurable with respect to $F_{t-m-1}$. Since $(\zeta_{j})$ is a sequence of independent random variables, this implies that $\zeta_{t-m}$ is a.s. constant when $\zeta_{t-m}\geq 0$ or when $\zeta_{t-m}< 0$, contradicting the hypothesis saying $\zeta_0^\delta$ has a non-degenerate distribution. This achieves (i).\\
(ii) The representation \eqref{di} is the same as
\begin{eqnarray*}
  \sigma_t^\delta &=&  b_0+ \Psi^+(L) (\max(X_{t},0))^\delta +\Psi^-(L) (\max(-X_{t},0))^\delta.
\end{eqnarray*}
with $ \Psi^+={\Upsilon_{\theta_1}^{-1}}{\Delta_{\theta_2}^+}$, $\Psi^-={\Upsilon_{\theta_1}^{-1}}{\Delta_{\theta_2}^-}$ and $\Delta_{\theta_2}^+(L)=\sum_{i=1}^p\alpha_i(1-\gamma_i)L^i, \Delta_{\theta_2}^-(L)=\sum_{i=1}^p\alpha_i(1+\gamma_i)L^i$ and $\Upsilon_{\theta_1}(L)=\sum_{i=1}^q\beta_i L^i$,  where $(\Delta_{\theta_2}^+,\Upsilon_{\theta_1})$ and $(\Delta_{\theta_2}^-,\Upsilon_{\theta_1})$ respectively coprime and $\theta_1=(\beta_i)_{1 \leq i \leq q}$, $\theta_2=\big((\alpha_i)_{1 \leq i \leq p},(\gamma_i)_{1 \leq i \leq p}\big)$, then $\theta=(\omega,\theta_1,\theta_2)$.\\
Suppose that there exist others polynomials $\Delta_{\theta_2'}^{+}=\sum_{i=1}^p\alpha_i'(1-\gamma_i')L^i,~ \Delta_{\theta_2'}^{-}=\sum_{i=1}^p\alpha_i'(1+\gamma_i')L^i, ~\Upsilon_{\theta_2'}=\sum_{i=1}^q\beta_i' L^i$ satisfying $\Psi^+=\Upsilon_{\theta_1'}^{-1}{\Delta_{\theta_2'}^{+}}, ~\Psi^-=\Upsilon_{\theta_1'}^{-1}{\Delta_{\theta_2'}^{-}}$ with $(\Delta_{\theta_2'}^{+},~\Upsilon_{\theta_1'})$, $(\Delta_{\theta_2'}^{-},\Upsilon_{\theta_1'})$ respectively coprime. Then
\begin{center}
$\left\{
  \begin{array}{ll}
     {\Upsilon_{\theta_1}^{-1}}{\Delta_{\theta_2}^+}={\Upsilon_{\theta_1'}^{-1}}{\Delta_{\theta_2'}^{+}}\\
 \\
     {\Upsilon_{\theta_1}^{-1}}{\Delta_{\theta_2}^-}={\Upsilon_{\theta_1'}^{-1}}{\Delta_{\theta_2'}^{-}}
     \end{array}
\right.
\Rightarrow \left\{
  \begin{array}{ll}
     {\Delta_{\theta_2}^+}=\big(\Upsilon_{\theta_1}{\Upsilon_{\theta_1'}^{-1}}\big){\Delta_{\theta_2'}^{+}}\\
 \\
     {\Delta_{\theta_2}^-}=\big(\Upsilon_{\theta_1}{\Upsilon_{\theta_1'}^{-1}}\big){\Delta_{\theta_2'}^{-}}
     \end{array}
\right.$
\end{center}
from the first equality, since $\deg(\Delta_{\theta_2}^+)=\deg(\Delta_{\theta_2'}^{+})=q$, we conclude that $\Upsilon_{\theta_1}{\Upsilon_{\theta_1'}^{-1}}=1$, therefore $\Upsilon_{\theta_1}={\Upsilon_{\theta_1'}}$ and so ${\Delta_{\theta_2}^+}={\Delta_{\theta_2'}^{+}}$, likewise from the second equality we conclude that ${\Delta_{\theta_2}^-}={\Delta_{\theta_2'}^{-}}$.
\begin{itemize}
  \item The equalities $     {\Delta_{\theta_2'}^+}={\Delta_{\theta_2}^+}$, ${\Delta_{\theta_2'}^-}={\Delta_{\theta_2}^-}$ implies that $\alpha_i(1-\gamma_i)=\alpha_i'(1-\gamma_i')$ and $\alpha_i(1+\gamma_i)=\alpha_i'(1+\gamma_i')$ which give $\alpha_i=\alpha_i'$ and $\gamma_i=\gamma_i'$.
  \item The equality $\Upsilon_{\theta_1}=\Upsilon_{\theta_1'}$ implies that $\beta_i=\beta_i'$.
  \item Since $ (\beta_i)_{i=1,p}, b_0=w(1-\sum_{j=1}^p\beta_j)^{-1}$ are unique then  $\omega$ is unique.
\end{itemize}
Thus, Condition {\bf C3} is established and the proof of proposition is achieved.
\end{proof}

\begin{proof}[Proof of Proposition \ref{QMLArmagarch}]
Since we prove that Lemma \ref{cond2} implies that conditions on Liptshitzian coefficients $(\alpha_j^{(i)}(f,\Theta))_j$ and $(\alpha_j^{(i)}(M,\Theta))_j$, it remains to prove conditions {\bf C2} and  {\bf C3}. Condition {\bf C2} holds since $c_0$ is supposed to be a positive number. Finally, condition {\bf C3} also holds since $f_\theta=f_{\theta'}$ implies $\psi_j(\theta)=\psi_j(\theta')$ for all $j\in \Z$. Therefore the parameters of the ARMA part of the process are identified and then the identification of the parameters GARCH can be deduced from the proof of Proposition \ref{QMLAPARCH}.
\end{proof}
\begin{proof}[Proof of Proposition \ref{QMLArmaparch}]
This proofs mimics exactly the proof of Proposition \ref{QMLArmagarch}.
\end{proof}

\vskip0.2cm
\noindent {\bf Acknowledgements.} The authors are grateful to the referees and associated editor for their comments which helped to improve the contents of this paper.\\

\end{document}